\documentclass{amsart}[12pt]
\def\doctype{}

\usepackage{latexsym,amssymb}
\usepackage{fancyhdr}
\usepackage{enumitem}
\usepackage{hyperref}

\newcommand{\cA}{\mathcal{A}}
\newcommand{\cB}{\mathcal{B}}

\newcommand\Z{\mathbb{Z}}
\newcommand{\cF}{\mathcal{F}}
\newcommand{\cD}{\mathcal{D}}
\newcommand{\comment}[1]{}

\numberwithin{equation}{section}

\fancypagestyle{titlepage}{
\fancyhead[R]{\doctype}
\fancyhead[CL]{}
\cfoot{\vspace{5pt} \thepage}
}

\let\oldsection\section
\newcommand\boldsection[1]{\oldsection{\bf #1}}
\newcommand\starsection[1]{\oldsection*{\bf #1}}
\makeatletter
\renewcommand\section{\@ifstar\starsection\boldsection}
\makeatother

\newtheoremstyle{theorem}
  {12pt}		  
  {0pt}  
  {\sl}  
  {\parindent}     
  {\bf}  
  {. }    
  { }    
  {}     
\theoremstyle{theorem}
\newtheorem{thm}{Theorem}[section]  
\newtheorem{lemma}[thm]{Lemma}     

\newtheorem{prop}[thm]{Proposition}

\newtheoremstyle{definition}
  {12pt}		  
  {0pt}  
  {}  
  {\parindent}     
  {\bf}  
  {. }    
  { }    
  {}     
\theoremstyle{definition}

\newtheorem{ex}[thm]{Example}

\newcommand\rk{{\sc Remark.} }

\renewcommand{\proofname}{Proof}

\makeatletter
\renewenvironment{proof}[1][\proofname]{\par
  \pushQED{\qed}%
  \normalfont \partopsep=\z@skip \topsep=\z@skip
  \trivlist
  \item[\hskip\labelsep
        \scshape
    #1\@addpunct{.}]\ignorespaces
}{%
  \popQED\endtrivlist\@endpefalse
}
\makeatother

\makeatletter
\renewcommand*\@maketitle{%
  \normalfont\normalsize
  \@adminfootnotes
  \@mkboth{\@nx\shortauthors}{\@nx\shorttitle}%
  \global\topskip42\p@\relax 
  \@settitle
  \ifx\@empty\authors \else {\vskip 1em
\vtop{\centering\shortauthors\@@par}} \fi
  \ifx\@empty\@date \else {\vskip 1em \vtop{\centering\@date\@@par}}\fi 
  \ifx\@empty\@dedicatory
  \else
    \baselineskip18\p@
    \vtop{\centering{\footnotesize\itshape\@dedicatory\@@par}%
      \global\dimen@i\prevdepth}\prevdepth\dimen@i
  \fi
  \@setabstract
  \normalsize
  \if@titlepage
    \newpage
  \else
    \dimen@34\p@ \advance\dimen@-\baselineskip
    \vskip\dimen@\relax
  \fi
} 
\renewcommand*\@adminfootnotes{%
  \let\@makefnmark\relax  \let\@thefnmark\relax
  \ifx\@empty\@subjclass\else \@footnotetext{\@setsubjclass}\fi
  \ifx\@empty\@keywords\else \@footnotetext{\@setkeywords}\fi
  \ifx\@empty\thankses\else \@footnotetext{%
    \def\par{\let\par\@par}\@setthanks}%
  \fi
\thispagestyle{titlepage}
}
\makeatother

\begin{document}

\title[KTS with subdesigns]{\large An Update on the Existence of Kirkman Triple Systems with Subdesigns}

\author{Peter J.~Dukes}
\address{\rm Peter J.~ Dukes:
Mathematics and Statistics,
University of Victoria, Victoria, Canada
}
\email{dukes@uvic.ca}

\author{Esther R.~Lamken}
\address{\rm Esther R.~Lamken:
773 Colby Street, San Francisco, CA, USA 94134
}
\email{esther.lamken@gmail.com}

\date{\today}

\begin{abstract}
A Kirkman triple system of order $v$, KTS$(v)$, is a resolvable Steiner triple system on 
$v$ elements. In this paper, we investigate an open problem posed by Doug Stinson, namely the existence of KTS$(v)$ which contain as a subdesign a Steiner triple system of order $u$,  an 
STS$(u)$. We present several different constructions for designs of this form. As a consequence, we completely settle the extremal case $v=2u+1$, for which a list of possible exceptions had remained for close to 30 years.  
Our new constructions also provide the first infinite classes for  the more general problem.  
We reduce  the other maximal case $v=2u+3$ to (at present)  three possible exceptions.  
In addition, we obtain results for other cases of the form $v=2u+w$ and also near $v=3u$.  Our primary method introduces a new type of Kirkman frame which contains  group divisible design subsystems. These subsystems can occur with different configurations, and we use two different varieties in our constructions.
\end{abstract}

\subjclass[2020]{05B07 (primary); 05B05, 05B10 (secondary)}
\maketitle
\hrule

\section{Introduction}
\label{intro}

A \emph{Steiner triple system} is a pair $(V,\cB)$, where $V$ is a finite set of points, $\cB\subset \binom{V}{3}$ is a set of $3$-element subsets of $V$ called \emph{blocks}, and such that any two distinct points appear together in exactly one block.  The abbreviation STS$(v)$ is used to denote a Steiner triple system as above, where $v=|V|$ is its \emph{order}.  It is easy to see from basic counting that $v\equiv 1$ or $3 \pmod{6}$ is a necessary condition for  the existence of STS$(v)$.  Conversely, it is known \cite{Kirkman-STS} that STS$(v)$ exist for all such orders $v \ge 1$.  Steiner triple systems represent the first nontrivial case for block designs.

A \emph{subdesign} in a Steiner triple system $(V,\cB)$ is a pair $(U,\cA)$ where $U \subseteq V$, $\cA \subseteq \cB$, and $(U,\cA)$ is itself a Steiner triple system.
From the Doyen-Wilson theorem, \cite{DW}, there exists an STS$(v)$ containing an STS$(u)$ as a subdesign for $v>u$ if and only if $u$ and $v$ are both admissible orders for Steiner triple systems and $v \ge 2u+1$.  The inequality is necessary because an element in $V \setminus U$ occurs in blocks with every other point in pairs, but with at most one point of $U$ at a time.

A combinatorial design $\cD$ is \emph{resolvable} if the blocks of $\cD$ can 
be partitioned into parallel classes such that each element of $\cD$ is contained in 
precisely one block of each class. 
In more detail, for a resolvable Steiner triple system of order $v$, 
we have $|\cB|=v(v-1)/6$ blocks resolving into $(v-1)/2$ parallel classes, each of size $v/3$.  The term \emph{Kirkman triple system}, with abbreviation KTS$(v)$, is generally used for a resolvable Steiner triple system.  
Indeed, the famous 
Kirkman Schoolgirl problem, posed in 1850 \cite{Kirkman-schoolgirl}, 
requests the construction of a KTS$(15)$.  In Figure~\ref{KS15}, we display Kirkman's 
solution to the Kirkman Schoolgirl Problem, \cite[Example 22.1.3] {Anderson}.  It 
was constructed from what Kirkman called a `curious arrangement' which is now 
known as a Room square of side 8. Note that this design contains 
as a subdesign an STS$(7)$ defined on the set $\{1,2,3,4,5,6,7\}$.  This is the first 
nontrivial example of a Kirkman triple system  which contains as a subdesign a Steiner 
triple system.

\begin{figure}[htbp]
\begin{tt}
\begin{tabular}{ccccc}
123 & hi4 & kl5 & mn6 & op7 \\
147 & il2 & mo3 & np5 & hk6 \\
156 & no2 & hl3 & mp4 & ik7 \\
267 & lp1 & in3 & ko4 & hm5 \\
245 & im1 & kp3 & lo6 & hn7 \\
357 & ho1 & km2 & ln4 & ip6 \\
346 & kn1 & hp2 & io5 & lm7 
\end{tabular}
\end{tt}
\caption{Kirkman's KTS(15), with a subdesign on $\{1,2,\dots,7\}$}
\label{KS15}
\end{figure}

In this paper, we explore the existence question for KTS$(v)$ containing STS$(u)$ as  subdesigns.  This is listed as open problem \#4 in Doug Stinson's survey \cite{Stinson-KTS} on Kirkman triple systems.  A further discussion of the problem and its present status appears in \cite[Section 19.7]{CR}.
 In what follows, we assume $v \equiv 3\pmod{6}$, $u \equiv 1$ or $3 \pmod{6}$ and $v>u$.  In a KTS$(v)$, a `Kirkman subsystem' or sub-KTS$(u)$  is a subdesign which is resolvable in such a way that parallel classes of the KTS$(u)$ are restrictions of parallel classes of the KTS$(v)$.  We do not assume such extra structure for subdesigns here.  Nonetheless, it is known \cite{RS,Stinson-KTS} that a KTS$(v)$ with a sub-KTS$(u)$ exists if and only if $u,v \equiv 3 \pmod{6}$ and $v \ge 3u$.  This provides a partial existence result for our problem.

In \cite{MV}, Mullin and Vanstone consider the case  $v=2u+1$, that is a KTS$(v)$ 
containing a maximum subdesign STS$(u)$.  They generalized Kirkman's original construction 
 and used  Room square starters to provide a direct construction for KTS$(2u+1)$ 
with maximum subdesigns, STS$(u)$.  They also  established the PBD-closure of such designs. 
Mullin, Stinson and Vanstone consider this  case in greater detail in \cite{MSV},  
introducing  the notation MK$(v)$ for a KTS$(v)$ containing a maximum subdesign 
STS$(u)$ for $v=2u+1$. They showed using a finite-field construction, 
PBD-closure and other recursive constructions that an MK$(v)$ exists for all but $19$ possible values of $v$.  Stinson's survey \cite[Theorem 4.2]{Stinson-KTS}, reported the set of possible exceptions for MK$(2u+1)$ as having been reduced to
\[ \mathcal{E}=\{115, 145, 205, 265, 355, 415, 649, 655, 697, 1243\}. \]
Here, we update the problem by eliminating these remaining exceptions. 
 To take care of the cases not covered by PBD-closure, we introduce the idea of Kirkman frames 
which contain as subdesigns group divisible designs with block size $3$. 

\begin{thm}
\label{thm:mkts}
There exists an MK$(2u+1)$ for all $u \equiv 1 \pmod{6}$.
\end{thm}

This result provides KTS$(v)$ with maximum STS subdesigns for $v\equiv 3 \pmod{12}$. 
Similar methods can be used
for the case $v\equiv 9\pmod{12}$.  In this case, the largest possible Steiner triple 
system subdesign is an
STS$(u)$ where  $v=2u+3$.

\begin{thm}
\label{thm:2u+3}
There exists a KTS$(2u+3)$ which contains as a subdesign an STS$(u)$ for $u\equiv 3\pmod{6}$ 
and $u\geq 3$ except possibly for $u=3x$ for $x \in \{83,107,179\}$.
\end{thm}

In the next section, we describe preliminary results for group divisible designs and 
frames, and introduce the idea of frames which contain group divisible 
designs as subdesigns. We are primarily interested in the case when such subdesigns are maximal. The starter construction described in \cite{MV} can be generalized to provide a cyclic construction for Kirkman frames with 3-GDD subdesigns. 
Section~\ref{sec:frames} collects existence results for frames with maximum subdesigns.  
New direct constructions combined with PBD-closure and basic recursive constructions 
for frames are the main techniques used in this section.
The proofs of Theorems~\ref{thm:mkts} and \ref{thm:2u+3} are given in Section~\ref{sec:maximal}.
Then, in Section~\ref{sec:general}, we obtain some preliminary results on the existence problem for Kirkman triple systems with subdesigns which are not necessarily maximal.  
We conclude with some remarks on techniques which may be useful 
in future constructions and on computational challenges for
completing the spectrum.

\section{Preliminary results}

\subsection{GDDs and frames}

Let $T$ denote an integer partition of $v$.
A \emph{group divisible design} of \emph{type} $T$ with block sizes in $K$,
denoted by GDD$(v,K)$ of type $T$ or as a $K$-GDD of type $T$,  is a triple $(V,\Pi,\cB)$ such that
\begin{itemize}[topsep=0pt]
\item
$V$ is a set of $v$ points;
\item
$\Pi=\{V_1,\dots,V_u\}$ is a partition of $V$ into \emph{groups} so that
$T=(|V_1|,\dots,|V_u|)$;
\item
$\cB \subseteq \cup_{k \in K} \binom{V}{k}$ is a set of blocks meeting each
group in at most one point; and
\item
any two points from different groups appear together in exactly one
block.
\end{itemize}

It is convenient to use exponential notation for the type of a GDD. 
That is, we say that a  GDD has type $g_1^{u_1}g_2^{u_2} \dots g_t^{u_t}$ if there are $u_i$ 
groups of size $g_i$ for $i=1,2,\dots,t$.  A GDD with a single block size $k$ is often 
denoted simply as a $k$-GDD instead of $\{k\}$-GDD. 

Two special types of GDDs will be used in our recursvive constructions. 
A \emph{pairwise balanced block design} or 
PBD, denoted by PBD$(v,K)$, is a GDD$(v,K)$ of type $1^{v}$.  
A GDD is \emph{uniform} if all groups have the same type. A \emph{transversal design} 
TD$(k,n)$ is a uniform GDD which has block size $k$ and precisely 
$k$ groups of size $n$,  a $k$-GDD of type $n^k$.  Note that the blocks in this case 
are transversals of the partition. It is well known that a TD$(k,n)$ is 
equivalent to a set of $k-2$ mutually orthogonal latin squares (MOLS) of order $n$.  
A resolvable TD$(k,n)$ is denoted by RTD$(k,n)$ and is equivalent to a TD$(k+1,n)$.
We refer the reader to \cite{Handbook} for results on the existence of transversal designs.

A $K$-frame of type $T$is a $K$-GDD of type $T$ whose block set can be 
partitioned into partial resolution classes such that each class is  a partition of 
$V\setminus V_i$ for some $V_i \in \Pi$.  
The same notation that we used  for GDDs is used for frames.   So a
frame with a single block size $k$ is called a $k$-frame.  In particular, $3$-frames can be used to construct Kirkman triple systems and for this reason are often  called Kirkman frames. 
In what follows, we restrict our attention to the case of block size three.
We recall here that necessary and sufficient conditions are known for both $3$-GDDs and Kirkman frames. 

\begin{thm}[\cite{Hanani}]
\label{3GDD}
The necessary and sufficient conditions for the existence of a $3$-GDD of type 
$g^u$ are $u\geq 3$, $g(u-1) \equiv 0 \pmod{2}$, and $u(u-1)g^2\equiv 0 \pmod{6}$.
\end{thm}

\begin{thm}[\cite{Stinson-Kirkframe}]
There exists a Kirkman frame of type $t^u$ if and only if $t$ is even, $u\geq 4$, and 
$t(u-1)\equiv 0 \pmod{3}$.
\end{thm}

We are interested in Kirkman frames which contain $3$-GDDs as subdesigns.  
Suppose $\cF$ is a uniform Kirkman frame of type $h^u$ defined on a set $V$ with 
groups $V_i$ for $i=1,2,\dots,u$ and block set $\cB$. Let $\cD$ be a $3$-GDD of 
type $g^u$ defined on a set $W$ with groups $W_i$ for $i=1,2,\dots, u$ and 
block set $\cA$. $\cD$ is a subdesign of $\cF$ if $W_i \subseteq V_i$ for 
$i=1,2,\dots, u$ and $\cA \subseteq \cB$.  
We call $\cF$ a Kirkman frame of \emph{type} $(g;h)^u$ to emphasize the subdesign.  
Note that $\cD$ is a maximum subdesign of $\cF$ if $h=2g$. 
The argument is similar to that in  Section~1, considering an element outside the subdesign, 
and further details in the more general case of IGDDs can be found in \cite{DLL}.

There is a close connection between Kirkman triple systems  
with maximum subdesigns and Kirkman frames with maximum $3$-GDD subdesigns.
Suppose there exists an MK$(2u+1)$ defined on a set $V$ with the subdesign defined 
on $V^{\prime}$, $|V^{\prime}|=u$.  If we delete an element $x \in V\setminus V^{\prime}$, the 
resulting design is a Kirkman frame of type $(1;2)^{u}$.  It is easy to see that adjoining 
a new element to a Kirkman frame of type $(1;2)^{u}$  will give us an MK$(2u+1)$. This 
provides the first examples of Kirkman frames with maximum $3$-GDD subdesigns.

\begin{prop}
An MK$(2u+1)$ is equivalent to a Kirkman frame of type $2^u$ which 
contains as a subdesign a $3$-GDD of type $1^u$.
\end{prop}

Ray-Chaudhuri and Wilson \cite[Theorem 6]{RCWil-school} describe a finite field construction for 
Kirkman triple systems, and the construction inherently produces a maximal sub-STS. 
(This construction is also described in \cite{Stinson-KTS}.)

\begin{lemma}[\cite{RCWil-school}]
\label{lem:prime}
For all prime powers $q\equiv 1\pmod{6}$, there exists an MK$(2q+1)$ or a Kirkman 
frame of type $(1;2)^q$.
\end{lemma}

\subsection{Starters and cyclic designs}

In \cite{MV}, Mullin and Vanstone describe an algebraic construction for MK$(2u+1)$
 which uses  the Mullin-Nemeth strong starter for Room squares. (See \cite{Handbook,Bluebk} 
for surveys on Room squares.)
The underlying idea is to adjoin $u$ new elements to the pairs in a Room square of side $u$ to 
form triples and then complete the design by including  the blocks of a Steiner triple 
system on the set of  $u$ new elements. 
This idea can be extended to construct both Kirkman triple systems and  Kirkman frames having maximum subdesigns.
 
Let $G$ be an additive abelian group of order $g$, and let $H \le G$ be a subgroup of order $h$, where $g-h$ is even. A \emph{frame starter} in $G \setminus H$ is a set of 
unordered pairs $S = \{ \{s_i,t_i\} :  1 \leq i \leq (g-h)/2 \}$ such that

\begin{itemize}[topsep=0pt]
\item[(1)] $ \{s_i \,:\, 1\leq i\leq (g-h)/2\} \cup \{t_i \,:\, 1\leq i\leq (g-h)/2\} =G\setminus H$. 
\item[(2)] $\{\pm(s_i -t_i) : 1\leq i \leq (g-h)/2) \} = G \setminus H$
\end{itemize}
  
If $S$ has the additional property that  $s_i+ t_i =s_j +t_j$ implies $i=j$ and 
 $s_i + t_i \notin  H$ for all $i$, then $S$ is called a \emph{strong starter}.

An \emph{adder} for $S$ is an injective mapping $A: S \rightarrow G\setminus H$ such 
that $ \{ s_i + a_i : i=1,2,\dots (g-h)/2\}  \cup 
\{ t_i + a_i   : i=1,2,\dots,(g-h)/2 \} = G\setminus H$ where $A(s_i,t_i)=a_i$.
A strong starter has the property that $A = \{ -(s_i + t_i) : i=1,2,\dots,(g-h)/2\}$.
A frame starter in $G\setminus H$ and a corresponding adder generate a uniform Room 
frame of type $h^{g/h}$.  
Note that when $H= \{0\}$ and $|H|=h=1$,  a frame starter is known as a starter. 
Starters and adders can be used \cite{Bluebk} to construct Room squares, or equivalently Room frames of type $1^n$.

Let $V= G\times \{0,1\}$. A starter for a Kirkman frame of type $(h;2h)^{(g/h)}$ 
in $(G\setminus H) \times \{0,1\}$ is a set of triples,  
$S_1= \{ \{ (x_i,1), (s_i,0), (t_i,0) \} : i=1,2,\dots,(g-h)/2 \}  \cup T$,  with the following properties:
\begin{itemize}
\item[(1)] $S =\{ \{ (s_i,0), (t_i,0) \} : 1 \leq i \leq (g-h)/2 \}$ is a frame starter in
 $(G\setminus H)  \times \{0\}$.

\item[(2)]  Every element of $G \setminus H$ occurs precisely once in 
the set $ \{ s_i -x_i: i=1,2,\dots,(g-h)/2\} \cup \{ t_i -x_i : i=1,2,\dots (g-h)/2 \}.$

\item[(3)] $T =\{T_i : i=1,2,\dots, (g-h)/6\}$ is a set of 
base blocks in $(G\setminus H) \times \{1\}$ for a $3$-GDD of type $h^{g/h}$.

\item[(4)] Every element of  $(G \setminus H) \times \{0,1\}$ occurs precisely once in $S_1$. 
\end{itemize}

After developing under the group $G$, the starter
$S_1$ generates a Kirkman frame of type $(2h)^{(g/h)}$ which 
contains as a subdesign a $3$-GDD of type $h^{(g/h)}$ on $V$.
Note that in the case that the underlying frame starter $S$ is a strong 
starter, we have $x_i = -a_i = s_i + t_i$.  

The following example provides a simple illustration of this construction. 

\begin{ex}
The set $S = \{ \{ 4,6\}, \{1,5\}, \{2,3\} \}$ is a strong starter over $\Z_{7}$, and therefore 
a corresponding adder is  $A=\{4,1,2\}$.   
The resulting Kirkman frame of type $(1;2)^7$ defined on $\Z_7 \times \{0,1\}$ is generated by
$$S_1= \{ \{3_1, 4_0,6_0\}, \{6_1,1_0,5_0\},\{5_1,2_0,3_0\},\{1_1,2_1.4_1\}\},$$
where for convenience we abbreviate $(y,i)$ by $y_i$ for $y \in \Z_7$ and $i=0,1$.  
\end{ex}

\subsection{Howell designs}

Howell designs \cite{Bluebk}  can also be used  to construct Kirkman triple systems with 
subdesigns. This construction produces designs with STS and KTS subdesigns which are 
disjoint. It is our first direct construction where the STS subdesign is not necessarily 
maximal. Our method is similar to that used for Room squares; 
we  adjoin a new set of $u$ elements to a Howell design  of side $u$  along the  rows (or columns) and then use the blocks of an STS$(u)$ on the new set of elements 
to form resolution classes. Since the Howell design 
is missing pairs, an additional step is needed to  form triples to cover these `missing' pairs and 
combine these with blocks of the STS$(u)$ to form additional resolution classes. 

In particular, suppose $u\equiv w \equiv 3 \pmod{6}$.  We describe a construction for a KTS$(2u+w)$ defined 
on a set $U\times \{0,1\} \cup W$ where $|U|=u$ and $|W|=w$.
Let $A=U \times\{0\}$  and $B=U\times\{1\}$. 
The first step is to construct an H$(u,u+w)$ on $A \cup W$ with the properties that it 
is missing all pairs in $W$ and all pairs in $\frac{w-1}{2}$ parallel classes of triples on $A$.
Next, adjoin  $B$ as a set of $u$ new elements along the rows (or columns) to form triples. 
Each column (respectively row) will be missing $\frac{u-w}{2}$ elements from $B$; these are used 
to construct $\frac{u-w}{6}$ triples from an STS$(u)$ defined on $B$. 
A count of pairs and triples tells us that we have the right number of 
triples remaining to construct $\frac{w-1}{2}$ parallel classes on $A$ and 
$\frac{w-1}{2}$ parallel classes  on $B$. These are  combined with 
a KTS$(w)$  defined on $W$ to provide $\frac{w-1}{2}$ additional classes.
The resulting design will be a KTS$(2u+w)$ 
which contains as disjoint subdesigns an STS$(u)$ defined on $B$ and a KTS$(w)$ defined on $W$.  

We give two examples to illustrate this method. These examples will be used later in 
recursive constructions.

\begin{ex}
\label{ex:21}
A KTS$(21)$ with a sub-STS$(9)$ can be constructed by adjoining 9 new elements, 
$0_0, 0_1,0_2, 1_0,1_1,1_2, 3_0,3_1,3_2$, to the nonempty cells of the rows 
of an H$(9,12)$ as indicated  in Figure~\ref{howell}.  
 The triples in the last row complete the nine 
resolution classes given by the columns.   The last resolution class consists 
of blocks of the form $\{i_0,i_1,i_2\}$ for $i=0,1,\dots,6$.
The resolution classes of this design are listed explicitly in the Appendix.  
In Lemma~\ref{smallcases}, we describe an alternative way of viewing this interesting design.
\end{ex}

\begin{figure}[htbp]
\small
$$
\begin{array}{l|c|c|c|c|c|c|c|c|c|} \cline{2-10}
0_0 &4_25_1 &             & 6_02_0 & 4_15_2 &             &4_05_0 &              & 6_22_1 & 6_12_2 \\
\cline{2-10}
0_1 &            & 4_15_1 & 4_25_0 & 6_02_2 & 6_12_1 &            & 6_22_0 &              & 4_05_2 \\
\cline{2-10}
0_2 & 6_22_2 & 6_02_1 &             &             & 4_05_1 & 6_12_0 & 4_15_0 & 4_25_2 &              \\   \cline{2-10}
1_0 &             & 2_04_0 & 2_24_1 & 2_14_2 &             & 5_16_2 & 5_26_1 & 5_06_0 &              \\
\cline{2-10}

1_1 &2_14_1 & 5_06_2 &              & 5_16_1 & 5_26_0 &               &             & 2_24_0 & 2_04_2\\ \cline{2-10}

1_2 &5_06_1 &              & 5_26_2 &             & 2_04_1 & 2_24_2 & 2_14_0 &               & 5_16_0 \\ \cline{2-10}

3_0 &4_06_0 & 4_26_1 &              & 2_05_0 &              & 2_15_2 & 2_25_1 &              &4_16_2 \\\cline{2-10}

3_1 &2_05_2 &              & 2_15_1 & 4_06_2 & 2_25_0 &               & 4_26_0 & 4_16_1 &            \\\cline{2-10}

 3_2 &           & 2_25_2 & 4_06_1 &              & 4_26_2 & 4_16_0 &               &2_05_1 & 2_15_0 \\\cline{2-10} 
\multicolumn{1}{c}{} & 
\multicolumn{1}{c}{0_11_03_2} & 
\multicolumn{1}{c}{0_01_23_1} & 
\multicolumn{1}{c}{0_21_13_0} & 
\multicolumn{1}{c}{0_21_23_2} & 
\multicolumn{1}{c}{0_01_03_0} & 
\multicolumn{1}{c}{0_11_13_1} & 
\multicolumn{1}{c}{0_01_13_2} & 
\multicolumn{1}{c}{0_11_23_0} & 
\multicolumn{1}{c}{0_21_03_1} 
\end{array}
$$
\medskip
\caption{A Howell design H$(9,12)$ equivalent to a KTS$(21)$ with sub-STS$(9)$}
\label{howell}
\end{figure}

\normalsize

\medskip

\begin{ex}
\label{ex:75}
A KTS$(75)$ which contains a sub-STS$(33)$ and a disjoint sub-KTS$(9)$ can be constructed from an H$(33,42)$ and a cyclically-generated 
STS$(33)$ with short orbit blocks for difference $11$.  Here, the `extra' 4 resolution classes are constructed from the short orbit and 
a carefully chosen full orbit which partitions into 3 parallel classes of the STS$(33)$. 
\end{ex}

\rk
The first of these is simple enough to produce by hand; however, the latter is found with a relatively fast hill climbing algorithm.  For our results to follow, Examples~\ref{ex:21} and \ref{ex:75} are sufficient.  But we remark that larger examples of a similar structure can be found with more computing resources.  In general, hill climbing algorithms are known to be efficient methods for constructing strong starters, \cite{DinStin-hill},  Steiner triple systems,
\cite{Stinson-hill}, and Howell designs, \cite{DinStin-howellhill}.   Hill climbing 
is the basis for most of our direct computer constructions in the Appendix.

\section{Kirkman frames with maximum $3$-GDD subdesigns}
\label{sec:frames}

In this section, we collect  constructions and existence results for Kirkman frames with 
maximum $3$-GDD subdesgins.  
First, we observe that the direct product construction for frames \cite{Stinson-Kirkframe}
can be applied to inflate the size of a frame and its subdesign.

\begin{prop}
\label{Kfr-dp}
If there exists a Kirkman frame of type $(g;2g)^n$ and a resolvable TD$(3,m)$, 
then there exists a Kirkman frame of type $(gm;2gm)^n$.
\end{prop}
 
Two other well known  recursive constructions for  frames, 
`filling in the holes' and `pulling out a group', can also be used to construct Kirkman frames with maximum $3$-GDDs, 
\cite{Stinson-Kirkframe}. We will use the following result.

\begin{prop}
\label{Kfr-addgroup}
If there exists a Kirkman frame of type $(gm;2gm)^n$ and a
Kirkman frame of type $(g;2g)^{m+1}$, then there is a Kirkman frame of type 
$(g;2g)^{mn+1}$.   
\end{prop}

It is well known that Wilson's Fundamental Construction (WFC)  can be used to construct 
Kirkman frames, \cite{Stinson-Kirkframe}. We  apply  \cite[Construction 4.4]{MSVW-frames} 
to show that applying WFC with Kirkman frames with maximum $3$-GDD subdesigns 
as ingredients will result in a Kirkman frame with a maximum $3$-GDD subdesign.

\begin{thm}
\label{PBD-frames}
If there exists a PBD$(v,K)$ such that for every $k \in K$ there exists 
a Kirkman frame of type $(g;2g)^k$, then there exists a Kirkman frame of 
type $(g;2g)^v$.
\end{thm}

\begin{proof} We briefly sketch the proof to indicate the placement of the 
$3$-GDD subdesign. 
Suppose $(X,\cB)$ is a PBD$(v,K)$ where $|X|=v$.  We construct 
a Kirkman frame of type $(g;2g)^v$ defined on 
$X \times (\{1,2,\dots,g\} \cup \{1^{\prime}, 2^{\prime},\dots,g^{\prime}\})$ 
as follows. 
Let $B= \{a_1,a_2,\dots,a_k\}$ be a block of the PBD. Replace each block $B \in \cB$ 
with a Kirkman frame of type $(g;2g)^k$ defined on the set 
$B \times (\{1,2,\dots,g\} \cup \{1^{\prime}, 2^{\prime},\dots,g^{\prime}\})$ 
where the holes (groups) of the frame are $a_i \times 
(\{1,2,\dots,g\} \cup \{1^{\prime}, 2^{\prime},\dots,g^{\prime}\})$ for $i=1,2,\dots,k$ 
and the groups of the $3$-GDD are $a_i \times \{1^{\prime}, 2^{\prime},\dots,g^{\prime}\})$
for $i=1,2,\dots,k$. Note that the resulting $3$-GDD of type $g^v$ has 
groups $x \times \{1^{\prime}, 2^{\prime}, \dots, g^{\prime}\}$ for $x \in X$.
\end{proof}

Theorem ~\ref{PBD-frames} shows that Kirkman frames of type $(g;2g)^n$ are 
PBD-closed on the number of groups, $n$.  Kirkman frames of type $(3;6)^n$ 
can be used to construct KTS with maximum STS subdesigns.
A necessary condition for the existence of  these frames 
is $n\equiv 1\pmod{2}$.  We construct Kirkman frames of 
type $(3;6)^n$ for $n\geq 5$ and $n\equiv 1\pmod{2}$ with a small number of 
possible exceptions for $n$ using PBD-closure for the set $\{5,7,9\}$ and direct 
constructions for small values. We start with the current PBD-closure result.

\begin{thm}[see IV.3.23 in \cite{Handbook}]
\label{pbd-579}
There exists a PBD$(v,\{5,7,9\})$ for $v\geq 5$, $v\equiv 1 \pmod{2}$ 
with the definite exceptions  
$v \in \{11,13,15,17,19,23,27,29,31,33,39\}$ and 
the possible exceptions  
$v\in \{ 43,51,59,71,75,83,87,95,99,107,111,113,115,119,139,179\}.$
\end{thm}

\begin{lemma}
\label{Kfr-small}
There exist Kirkman frames of type $(3;6)^n$ for each $n\in 
\{5,7,9,11,13,15,17, 19,23,27$, $29,31,33,39,43,51,59,71,75,95,99, 113,115,119\}.$
\end{lemma}

\begin{proof}
If $n\equiv 1 \pmod{6}$ is a prime power, then by Lemma~\ref{lem:prime} there exists a Kirkman frame of 
type $(1;2)^n$. Expand using a $3$-RGDD of type $3^3$ to get a Kirkman frame 
of type $(3;6)^n$(Proposition~\ref{Kfr-dp}). This construction takes care of $n\in \{7,13,19,31,43\}.$ 
For $n\in \{5,9,11,15,17, 23,27,29,33,39,51,59\}$, 
we use Kirkman frame starters  of type $(3;6)^n$. The starters are listed in the
Appendix.

Proposition~\ref{Kfr-addgroup} is applied to take care of  two  cases. 
For $n=71$, start with 
a Kirkman frame of type $(1;2)^7$ and expand using an RTD$(3,30)$ to get 
a Kirkman frame of type $(30;60)^7$. Next apply Proposition~\ref{Kfr-addgroup} with $g=3$, 
$m=10$ and a Kirkman frame of type $(3;6)^{11}$ to construct a Kirkman frame of 
type $(3;6)^{71}$. 
For $n=113$, start with a Kirkman frame of type $(3;6)^7$ and expand using 
an RTD$(3,16)$ to get a Kirkman frame of type $(48;96)^7$.  
Apply  Proposition ~\ref{Kfr-addgroup} with $m=16$ and using a Kirkman frame of type $(3;6)^{17}$.  

The remaining five cases are done using PBDs with larger block sizes. There 
exist PBD$(v,K)$   for the following $(v,K)$ parameter sets: 
$(75,\{5,15\})$,  $(95, \{5,19\})$, $(99, \{9,11\})$, $(115, \{5,23\})$, $(119, \{7,17\})$.  
(In each case, the PBD is constructed from a TD$(k,m)$ where $K=\{k,m\}$.) 
We apply Theorem~\ref{PBD-frames} with $g=3$ to construct Kirkman frames of 
type $(3;6)^n$ for $n\in \{75, 95,99,115,119\}$. 
\end{proof}

Combining Theorem~\ref{pbd-579} and Lemma~\ref{Kfr-small}
 results in the following existence result for Kirkman frames of type $(3;6)^n$.

\begin{thm}
\label{pbdclosure3,6}
There exists a Kirkman frame of type $6^n$ which contains as a subdesign a $3$-GDD 
of type $3^n$ for $n\equiv 1 \pmod{2}$ for $n\geq 5$ except possibly for 
$n \in \mathcal{E}_2 = \{ 83,87,107,111,139,179\}.$
\end{thm}

It is also useful to have some small Kirkman frames of other types.

\begin{lemma}
\label{Kfr-6,12small}
There exist Kirkman frames for the following types: 
\begin{itemize}[topsep=-5pt]
\item $(2;4)^7$
\item  $(6;12)^n$  for $n=5,7,8,9$   
\item  $(12;24)^n$ for $n=4,5,7,8,9$. 
\end{itemize}
\end{lemma}

\begin{proof} 
Kirkman frames of types $(2;4)^7$, $(6;12)^n$ for $n=5,8,9$, 
$(12;24)^4$, and $(12;24)^8$ are constructed 
directly using Kirkman frame starters; these starters are listed in the Appendix. 
A Kirkman frame of type $(6;12)^7$ is constructed using Proposition~\ref{Kfr-dp} 
and the Kirkman frame of type $(2;4)^7$ with  $m=3$.
Kirkman frames of types $(12;24)^n$ for $n=5,7,9$ 
are constructed using the Kirkman frame direct product, Proposition~\ref{Kfr-dp},  
using Kirkman frames of types $(3;6)^n$ with $m=4$.
\end{proof}

These small cases are enough to allow us to use PBD-closure to construct 
Kirkman frames of types $(6;12)^n$ and $(12;24)^n$ for all but a small number of 
exceptions.

\begin{thm}
\label{Kfr-12,24}
There exists a Kirkman frame of type $24^n$ which contains as a subdesign 
a $3$-GDD of type $12^n$ for $n$ a positive integer, $n\geq 4$ except possibly 
for $n \in \{6,10,12,14,18,26,30\}$. 
\end{thm}

\begin{proof}
Using the PBD-closure for the set $\{4,5,7,8,9\}$, \cite[IV.3.23]{Handbook}, 
there exist Kirkman frames of type $(12;24)^n$ for $n\geq 4$ and 
$n\notin \mathcal{E}_3 =\{6,10,11,12,14, 15,18,19,23,26,27,30,51\}$. Applying the frame 
direct product, Proposition~\ref{Kfr-dp}, with Kirkman frames of types $(3;6)^n$ 
with $n=11,15,19,23,27,51\}$ and $m=4$ takes care of the remaining odd values 
of $n \in \mathcal{E}_3$.
\end{proof}

PBD-closure for the set $\{5,7,8,9\}$, \cite[IV.3.23]{Handbook}  also gives us the 
following result.  We can use Theorem~\ref{PBD-frames} with $g=2$ and $k=7$ to 
take care of one of the exceptions, namely $n=43$. 

\begin{thm}
\label{Kfr-6,12}
There exists a Kirkman frame of type $12^n$ which contins as a subdesign a $3$-GDD of 
type $6^n$ for all integers $n \geq 4$, except possibly for 
$$n\in \{ 4,6, 10,\dots,20, 22,23,24, 26,\dots,34, 38,39, 42,44, 46,51, 52, 60, 
94,95,96, 98,99,100,102,104,$$  $$106, 107,108,110,111,116,138,140,142,146,150,154,
156,158,162,166,170,172,174,206\}.$$
\end{thm}

\rk
Theorems~\ref{Kfr-12,24} and ~\ref{Kfr-6,12} are sufficient for our results to follow.
It is likely that direct constructions for some of the smaller designs 
could be used to reduce the number of possible exceptions for these results. 
For example, if a Kirkman frame with a sub-GDD of type $(6,12)^4$ 
were available, this would reduce the list of possible exceptions  for 
Theorem~\ref{Kfr-6,12} to just 13 cases, 
$\mathcal{E}_3$.

\section{Maximum subdesigns}
\label{sec:maximal}

In this section, we treat the case of maximal subdesigns in Kirkman triple systems, proving Theorems~\ref{thm:mkts} and \ref{thm:2u+3}.  We consider each case in turn.

\subsection{Completing the spectrum for MK$(v)$}

Here, we let $u\equiv 1\pmod{6}$ and $v=2u+1\equiv 3\pmod{12}$. 
Let RMK = $\{u : \text{there exists an MK}(2u+1) \}$.  It is known \cite{MSV, MV} that RMK is PBD-closed; that is, if there exists a PBD$(v,K)$ such that for every $k\in K$ there exists an MK$(2k+1)$, 
then there exists an MK$(2v+1)$.  The main construction in \cite{MSV} uses Lemma~\ref{lem:prime}, the 
existence of MK$(2q+1)$ for $q$ a prime power congruent to $1 \pmod{6}$, together 
with PBD-closure of the set of such integers.

We begin with an updated PBD-closure result for prime powers congruent to $1\pmod{6}$. Note that this 
immediately settles two values listed as exceptions in \cite{Stinson-KTS}, namely $655,1243 \in \mathcal{E}$.

\begin{thm}[\mbox{\cite[IV.3.23]{Handbook}}]
\label{pbd1mod6}
Let $Q_{1(6)} := \{q: q \text{ is a prime power}, q \equiv 1 \pmod{6}\}$.
There exists a PBD$(v,Q_{1(6)})$ for all $v\equiv 1\pmod{6}$, except for $v=55$ and possibly $v\in \mathcal{E}_1$ 
where $$\mathcal{E}_1 =\{ 115, 145, 205,235,265,319,355,391,415,445,451,493,649,667,
685,697,745,781,799,805,1315\}.$$
\end{thm}

We  use Kirkman frames with maximum $3$-GDD subdesigns to construct most of 
the remaining cases.  The first construction is a direct product construction combined with 
`filling in the holes'.  For completeness, a brief proof is included to indicate how the
$3$-GDD subdesign is constructed.

\begin{thm}
\label{fr-dirprod}
If there exists a Kirkman frame of type $(2g)^n$ which contains as a subdesign a $3$-GDD 
of type $g^n$,  a resolvable TD($3,m)$, and an MK$(2gm+3)$, 
then there exists an MK$(2gmn+3)$.
\end{thm}

\begin{proof}
Let $\cF$ be a Kirkman frame of type $2g^n$ defined on $V=\cup_{i=1}^{n} V_i$ where 
$V_i = X_i \cup Y_i$, $|X_i|=|Y_i|=g$ and such that the $3$-GDD of type $g^n$ has 
groups $Y_i$, $i=1,2,\dots,n$. Expand using an RTD$(3,m)$ to get a Kirkman 
frame of type $(2gm)^n$ which contains as a subdesign a $3$-GDD of type $(gm)^n$. 
Note that the groups of this frame are $V_i \times \{1,2,\dots,m\}$ for $i=1,2,\dots,n$ 
and the groups of the GDD are $Y_i \times \{1,2,\dots,m\}$ for $i=1,2,\dots,n$. 
Let $M_i$ denote an MK$(2gm+3)$ defined on 
$V_i \times \{1,2,\dots,m\} \cup Z$, where $Z=\{\alpha,\beta,\infty\}$ and where an STS$(gm+1)$ subdesign exists on $Y_i \times \{1,2,\dots,m\} \cup \{\infty\}$. The design $M_i$ has 
$gm +1$ parallel classes; let them be denoted by $R_i^{j}$ for $j=1,2,
\dots, gm+1$. Without loss of generality, suppose $R_i^{gm+1}$ contains the block $Z$ for each $i$. 
Fill in each frame hole with the resolution classes $R_i^{j}$ for $j=1,2,\dots,gm$.  This gives 
us $gmn$ resolution classes. Another class is given by $\cup_{i=1}^{n} (R_i^{gm+1} \setminus 
\{Z\})  \cup \{Z\}$.
The subdesigns STS$(gm+1)$ fit in the groups of the $3$-GDD of type 
$(gm)^n$ to construct an STS$(gmn+1)$ subdesign. 
\end{proof}

\rk
Note that this  KTS$(2gmn+3)$ also contains as subdesigns KTS$(2gm+3)$ and STS$(gm+1)$.

Kirkman frames of types $(6;12)^n$ and $(12;24)^n$ can be used directly to construct Kirkman 
triple systems with maximum subdesigns.

\begin{lemma} 
\label{6,12-const}
 If there exists a Kirkman frame of type $(6;12)^n$, then there is a KTS$(12n+3)$ which 
contains as a subdesign an STS$(6n+1)$.  
\end{lemma}

\begin{proof}
Suppose there exists a Kirkman frame of type $(6;12)^n$  
defined on $\cup_{i=1}^{n} W_i$ where 
$W_i = X_i \cup Y_i$ and $|X_i| =|Y_i| = 6$. The groups of the frame are $W_i$ for 
$i=1,2,\dots,n$. The $3$-GDD is defined on $\cup_{i=1}^{n} Y_i$, with groups $Y_i$ for $i=1,2,\dots,n$.  Put $Z=\{\alpha,\beta,\infty\}$ as a set of new points.

We fill in the holes  of the frame with copies of a KTS$(15)$ defined on each $W_i \cup Z$, where $Z$ is a block of each.  Let $R'_i$ 
denote the class of the corresponding KTS$(15)$ which contains the block $Z$.
Put $R^{\prime} = \cup_{i=1}^{n} (R_i^{\prime} \setminus \{Z\}) 
\cup \{Z\}$. The remaining six parallel classes of the KTS are used 
to complete the $6n$ frame classes. This gives us $6n+1$ classes altogether, and thus a KTS$(12n+3)$.

The $3$-GDD of type $6^n$ has groups defined on $Y_i$ for $i=1,2,\dots,n$. Each of 
these groups gets filled in with an STS$(7)$ defined on $Y_i \cup \{\infty\}$ since 
each KTS$(15)$ contains such a subdesign. This results in an STS$(6n+1)$ which 
intersects each KTS$(15)$ in an STS$(7)$. 
\end{proof}

The construction is similar for Kirkman frames of type $(12;24)^n$. In this case, we 
fill in the holes of the frames using a KTS$(27)$ which contains as a subdesign an STS$(13)$.

\begin{lemma} 
\label{12,24-const}
  If there exists a Kirkman frame of type $(12;24)^n$, then 
there exists a KTS$(24n+3)$ which contains as a subdesign an STS$(12n+1)$. 
\end{lemma}

We are now in a position to prove Theorem~\ref{thm:mkts} and settle the 
remaining cases.   As in \cite{MSV}, we use PBD-closure for our main construction.  
This leaves us with the values in $\mathcal{E}_1$ from the  updated PBD-closure result, Theorem~\ref{pbd1mod6}. The case $u=55$ was added without proof as an 
addendum in \cite{MSV}.  We include  a recursive construction for this case. 

\begin{lemma}
There exists an MK$(2u+1)$ for $u \in \mathcal{E}_1 \cup \{55\}$.
\end{lemma}

\begin{proof} 
For $u \in \{55,319,355,391,415,445,451,493,745,481,799,805,1315\}$, we apply 
Lemma~\ref{6,12-const} using 
$n=9,53,59,65,69,74,75,82,124,130,133,134,$ and $219$ (respectively). 
Theorem~\ref{Kfr-6,12} provides the existence of the required $(6;12)^n$ Kirkman frames.
For $u \in \{205,265,649,697\}$, we apply Lemma~\ref{12,24-const} using  Kirkman frames of type $(12;24)^n$ for $n=17,22,54,$ and $58$, which come from Theorem~\ref{Kfr-12,24}. 

There are three special cases: $u=115,145,235$. The case $u=115$ 
is constructed directly using a strong starter in $\Z_u$ and cyclic STS$(u)$ on $\Z_u$. 
A strong starter used to construct this Kirkman frame of type $(1;2)^{115}$ is  listed in the Appendix.
 For $u=145$, we apply Theorem~\ref{fr-dirprod} with $g=12$, $n=4$, $m=3$, and 
an MK$(75)$ from Lemma~\ref{lem:prime}. Theorem~\ref{fr-dirprod} is also used for  $u=235$ with 
$g=1$, $n=13$, $m=18$, and an MK$(39)$ from Lemma~\ref{lem:prime}. 
\end{proof}

\rk In  \cite{MSV}, an indirect product construction was used to fill in the gaps left by PBD closure.  
The applications are given in several tables.  Instead of 
just filling in the remaining values of $\mathcal{E}$, we show how to use Kirkman 
frames with maximum $3$-GDD subdesigns and a single computer-generated construction 
to take care of the values in $\mathcal{E}_1$. Note that a complete existence result 
for Kirkman frames of type $(6;12)^n$ would also settle the existence of 
MK$(12n+3)$; see Theorem~\ref{Kfr-6,12} and Lemma~\ref{6,12-const}.

\subsection{The case $v=2u+3$}
\label{sec:2u+3}

Here, we construct KTS$(v)$ with maximum STS$(u)$ subdesigns in the congruence class $v\equiv 9 \pmod{12}$.  In this case, $v=2u+3$ where 
$u\equiv 3\pmod{6}$.
We first note the existence of the two smallest designs.

\begin{lemma}
\label{smallcases}
There exists an KTS$(2u+3)$ which contains as a subdesign an STS$(u)$ for $u=3$ and 
$u=9$.
\end{lemma}

\begin{proof}
For $u=3$, there exists a KTS$(9)$ and the subdesign is a single block of size $3$. 

For $u=9$, Example~\ref{ex:21} produces the required KTS$(21)$ by adjoining elements to an H$(9,12)$ Howell design.  See also Figure~\ref{howell}. We offer here an alternate viewpoint for this construction.  Begin with a Fano plane, expand each block by an RTD$(3,3)$.  It turns out that the set of blocks so produced can be resolved; see the Appendix for an explicit resolution.
Notice that the first 9 classes listed can be decomposed into 7 disjoint RTD$(3,3)$s, where each of these classes  contains one block from each RTD$(3,3)$. The last parallel class comes from the groups of these TDs.  A similar construction appears in \cite{KO} and this design is also among those listed in the compilation \cite{CCIL} of KTS$(21)$ with nontrivial automorphism group.  
\end{proof}

Our first recursive construction is an easy application of Wilson's Fundamental Construction,
\cite{Wilsonconst}.

\begin{prop}
\label{WFC-triple}
If there exists a KTS$(v)$ which contains as a subdesign an STS$(u)$, then there exists a KTS$(3v)$ which contains as a subdesign an STS$(3u)$.
\end{prop} 

\begin{proof}
Apply Wilson's Fundamental Construction, giving weight 3 and replacing each block of the KTS with 
an RTD$(3,3)$.  Parallel classes of the resulting design are induced from those of the KTS$(v)$ and RTDs.
\end{proof}

Note that when $v=2u+1$, the resulting design is a KTS$(6u+3)$ which contains 
as a subdesign an STS$(3u)$.

The key recursive construction for this case uses Kirkman frames of type $(3;6)^n$. 

\begin{thm}
\label{mainframe}
If there exists a Kirkman frame of type $6^n$ which contains as a subdesign a $3$-GDD 
of type $3^n$, then there is a KTS$(6n+3)$ which contains as a subdesign an STS$(3n)$.
\end{thm}

\begin{proof}
Let $V = \cup_{i=1}^n V_i $ where $V_i = X_i \cup Y_i$ and $|X_i|=|Y_i|=3$
for $i=1,2,\dots,n$. 
Suppose $\cF$ is a Kirkman frame of type $6^n$ defined on $V$ where the $3$-GDD has 
groups $\cup Y_i$ for $i=1,2,\dots,n$.  Adjoin a set of 3 new elements, $Z=\{\alpha,\beta,\infty\}$, 
and fill in the holes of the frame with an RTD$(3,3)$ defined on
$X_i \cup Y_i \cup Z$.  This results in a resolvable $3$-GDD of type $3^{2n+1}$ 
which contains as a subdesign a $3$-GDD of type $3^n$.  Fill in the groups of this design as an additional resolution class.  The result is a KTS$(6n+3)$ with a maximum subdesign STS$(3n)$.
\end{proof}

\rk This construction actually produces a $\Diamond$-type design, \cite{RS}.  That is, the 
resulting KTS$(6n+3)$ contains as subdesigns a KTS$(9)$ and an STS$(3n)$ which 
intersect in a single block $Y_i$. 

The next construction uses a tripling construction for Kirkman frames. 

\begin{lemma}
\label{triple}
If there exists a Kirkman frame of type $(3;6)^n$, then there is a KTS$(18n+3)$ 
which contains as a subdesign an STS$(9n)$.
\end{lemma}

\begin{proof}
Give weight 3 and replace blocks with an RTD$(3,3)$ to produce a Kirkman frame of 
type $(9;18)^n$. Add 3 new elements and fill the groups of the frame with a
KTS$(21)$ which contains as a subdesign an STS$(9)$ (see Lemma~\ref{smallcases})   
where the STS$(9)$ are aligned on the groups of the sub-GDD of type $9^n$. 
\end{proof}

We are now in a position to prove Theorem~\ref{thm:2u+3}.  Using the existence 
of Kirkman frames of type $(3;6)^n$, Theorem~\ref{pbdclosure3,6}, together 
with the main frame construction, Theorem~\ref{mainframe},  we can 
construct KTS$(6u+3)$ which contain as a subdesign an STS$(3u)$ for $u\geq 5$, 
$u\equiv 1\pmod{2}$ and $u\notin \mathcal{E}_2$. Lemma~\ref{smallcases} 
takes care of $u < 5$.  This leaves 6 cases to consider. 

\begin{lemma}
There exist KTS$(6u+3)$ which contain as a subdesign an STS$(3u)$ for
$$u \in \{87,111,139\}.$$
\end{lemma}

\begin{proof} For $u=87$ and $u=111$, apply the Kirkman frame 
tripling construction, Lemma~\ref{triple} with $n=29$ and $n=37$ respectively. 
The largest case, $u=139\equiv 1 \pmod{6}$, is handled by Theorem~\ref{thm:mkts} and Proposition~\ref{WFC-triple}.
\end{proof}

Our results here leave only three remaining open cases, namely KTS$(6u+3)$ which contain STS$(3u)$ for $u \in \{ 83, 107, 179\}$. 
The existence of two $\Diamond$-type designs could be used to complete these cases. 
If there exists a KTS$(69)$ which contains as subdesigns a KTS$(21)$ and an STS$(33)$ 
which intersect in an STS$(9)$, then there is a KTS$(6u + 3)$ which contains 
an STS$(3u)$ for $u=107$.  If there exists a KTS$(117)$ which contains as subdesigns 
a KTS$(21)$ and an STS$(57)$ intersecting in an STS$(9)$, then the other two cases, 
$u=83$ and $179$, could also be handled.

\section{More general subdesigns}
\label{sec:general}

\subsection{Subdesigns at a fixed offset from maximality}
\label{sec:gap}

An MK$(2v+1)$ is extremal in the sense that the gap, call it $w$, between the KTS order and double the subsystem order is as small as possible, namely $w=1$.  We have also considered the case $w=3$.  In this section, we offer some preliminary constructions and results for more general values of $w$.

We begin with a product construction that achieves a prescribed gap value.

\begin{lemma}
\label{flat}
If there exists an MK$(2v+1)$, an RTD$(3,w)$, and an STS$(w)$ with chromatic index at most $\frac{1}{2}(w-1)+v$, then there exists a KTS$(2vw+w)$ containing a sub-STS$(vw)$.
\end{lemma}

\begin{proof} 
Apply Wilson's Fundamental Construction, giving every element of the MK$(2v+1)$ weight $w$ and replacing blocks with RTD$(3,w)$.
The result is a resolvable GDD of type $w^{2v+1}$ with a sub-GDD of type $w^v$.  Fill groups with STS$(w)$, so that the result is an STS$(2vw+w)$ containing an STS$(vw)$ subdesign.  It remains to resolve the blocks of the larger STS.  For this, we follow a similar strategy as in \cite[Lemma 2.1]{CDLL}, which we summarize below.     By our assumption on the chromatic index of the STS$(w)$, there exists a resolution of blocks into partial parallel classes such that every element is missed equally often and at most $v$ times. The MK$(2v+1)$ induces $v$ parallel classes which are `flat' with respect to the direct product. 
So, we may break up sufficiently many of these classes to combine with partial parallel classes of the STS$(w)$ in each group.
\end{proof}

\rk
Assuming $v \ge 7$, the assumption on chromatic index can be achieved using a KTS$(w)$ if $w \equiv 3 \pmod{6}$, a `Hanani triple system' if $w \equiv 1 \pmod{6}$, $w \ge 19$, or directly for $w \in\{7,13\}$; see \cite[II.2.79-83]{Handbook}.

We next give a singular style product construction that yields some different instances of KTS$(2u+w)$ having subdesigns of order $u$. This construction uses a different type of 
GDD subdesign where the groups of the frame and the GDD subdesign coincide.
Suppose $\cF$ is a uniform Kirkman frame of type $h^u$ defined on a set $V$ with 
groups $V_i$ for $i=1,2,\dots,u$ and block set $\cB$. Let $\cD$ be a $3$-GDD of 
type $h^t$ defined on $ \cup _{i=1}^{t} V_i$ with groups $V_i$ for $i=1,2,\dots, t$ and 
block set $\cA$. $\cD$ is a subdesign of $\cF$ if  $\cA \subseteq \cB$.  
Note that in this case, $h\equiv 0 \pmod{2}$ and $u\geq 2t+1$. 
This type of Kirkman frame with GDD subdesign  arises naturally from KTS with 
maximum STS subdesigns. For example, deleting an element from the STS$(7)$ in an 
MK$(15)$ results in a Kirkman frame of type $2^7$ which contains as a subdesign 
a $3$-GDD of type $2^3$. 

\begin{lemma}
\label{sip1}
Suppose $v \equiv 1 \pmod{6}$, so that there exists an MK$(2v+1)$.

(a)  For $w \equiv 1 \pmod{6}$, $w \ge 7$, 
there exists a KTS$(v(w+1)+1)$ containing a sub-STS of order $\frac{1}{2}(v-1)(w+1)+1$ and a sub-KTS$(w+2)$ intersecting in one element.

(b) For $w \equiv 3 \pmod{6}$ and $z \in \{3,9,15\}$ with $w \ge z$ and $w+z \ge 18$, 
there exists a KTS$(v(w+z)+z)$ containing a sub-STS of order $\frac{1}{2}(v-1)(w+z)+z$ and sub-KTS$(w+2z)$ intersecting in a sub-KTS$(z)$.
\end{lemma}

\begin{proof} 
(a)
Delete a point from the STS$(v)$ subdesign of an MK$(2v+1)$, producing a Kirkman frame of type $2^v$ containing a sub-GDD of type $2^{(v-1)/2}$.  Give every point weight $(w+1)/2$ and replace blocks with RTD$(3,(w+1)/2)$.  The result is a Kirkman frame of type $(w+1)^v$ having a sub-GDD of type $(w+1)^{(v-1)/2}$.
Add one new element $\infty$ and fill groups with KTS$(w+2)$ having a common element $\infty$.

(b) 
We begin as in (a), except with weight $(w+z)/2$. Note that our assumption implies that this weight is at least $9$, and hence an RTD$(3,(w+z)/2)$ exists.  This gives a Kirkman frame of type $(w+z)^v$ having a sub-GDD of type $(w+z)^{(v-1)/2}$.
Add a set $Z=\{\infty_i: i=1,\dots,z\}$ of new elements and fill groups with KTS$(w+2z)$ having a   common subsystem on $Z$. We note that such a KTS exists since $w \ge z$.
The result is a KTS of order $v(w+z)+z$ having a sub-STS of order $\frac{1}{2}(v-1)(w+z)+z$ (from the sub-GDD), a sub-KTS of order $w+2z$ (from one group outside this GDD) and these subdesigns intersect in the subdesign on $Z$, a KTS$(z)$.
\end{proof}

We demonstrate part (a) of this construction with a bound for gap size $w=7$.  For this proof and some to follow, it is helpful to extend the exponential notation for frames with sub-GDDs to the nonuniform setting $(g_1;h_1)^{u_1} \cdots (g_t;h_t)^{u_t}$.

\begin{prop}
\label{gap7}
For all integers $u \equiv 1\pmod{6}$, $u \ge 2551$, there exists a KTS$(2u+7)$ containing a sub-STS$(u)$.
\end{prop}

\begin{proof}
Begin with a TD$(9,n)$.  Truncate all but $p$ points in the 8th group and all but $q$ points in the 9th group.  Give weights $(6;12)$ to all remaining points, setting up simultaneous applications of Wilson's fundamental construction.  Replace the blocks, which have sizes in $\{7,8,9\}$, with Kirkman frames having subdesigns of type $(6;12)^x$ for $x \in \{7,8,9\}$.
The result is a Kirkman frame with a sub-GDD of type $(6n;12n)^7 (6p;12p)^1 (6q;12q)^1$. 
Add a set of 9 new elements with 1 of these points assigned  to the subdesign.

From Lemma~\ref{sip1}(a) with $w=7$, there exists a KTS$(48t+9)$ containing a sub-STS$(24t+1)$ and a sub-KTS$(9)$ intersecting in one element.   Put $n=4t$ and fill the first seven groups of the above frame with these designs, aligning as usual the sub-STS on the groups of the sub-GDD and the sub-KTS$(9)$ on the new elements.  Similarly, put $p=4s$ and fill the 8th group with a KTS$(48s+9)$ containing a sub-STS$(24s+1)$.

From Lemma~\ref{flat}, there exists a KTS$(84r+21)$ containing a sub-STS$(42r+7)$ for each integer $r \ge 0$.  Letting $q=7r+1$, we fill the last group, plus extra points, with a KTS$(12q+9)$ containing a sub-STS$(6q+1)$. (Note that there is no condition needed on existence of a sub-KTS$(9)$ for this last group.)

The existence of a TD$(9,n)$ is guaranteed \cite[III.3.81]{Handbook} for $n > 570$, which implies our construction succeeds for all admissible $u > 25000$.   However, many smaller values of $n$ admit a TD$(9,n)$; see the table at \cite[III.3.87]{Handbook}.  Using  a computer to check below the guarantee, we have verified that any $u \ge 2551$, $u \equiv 1 \pmod{6}$, admits a representation as $u=168t+24s+42r+7$, where $n=4t \ge 4s,7r$ and there exists a TD$(9,n)$.
\end{proof}

\rk
Our bound on $u$ could be considerably lowered if a Kirkman frame with sub-GDD of type $(6;12)^4$ were available.  In this case, a TD$(5,n)$ could be used in place of a TD$(9,n)$.  

The case $w=9$ can be handled similarly, but with different ingredient designs.

\begin{prop}
\label{gap9}
For all integers $u \equiv 3 \pmod{6}$, $u \ge 21429$, there exists a KTS$(2u+9)$ containing a sub-STS$(u)$.
\end{prop}

\begin{proof}
As in the proof of Proposition~\ref{gap7}, we work from a frame of type $(12n)^7 (12p)^1 (12q)^1$ having a sub-GDD of type $(6n)^7 (6p)^1 (6q)^1$.  
Add 27 new elements, 9 of which are assigned to the subdesign.

Use Lemma~\ref{sip1}(b) with $v = 6t+1$ and $w=z=9$ to produce a
KTS$(108t+27)$ containing a sub-STS$(54t+9)$ and sub-KTS$(27)$ intersecting in an STS$(9)$.
Put $n=9t$ and fill the first seven groups of our frame with these designs, aligning the subdesigns as usual.  Similarly, put $p=9s$ and fill the 8th group with a KTS$(108s+27)$ containing a sub-STS$(54s+9)$ and sub-KTS$(27)$ intersecting in an STS$(9)$.

We fill the 9th group with a design constructed as follows.  Begin with a 
frame of type $(3;6)^{2r+1}$, where $r \in \{2,3,4,\dots,12\}$, from Lemma~\ref{Kfr-small},
and expand using an RTD$(3,11)$.
Example~\ref{ex:75} (see also the Appendix) gives a direct construction of a
KTS$(75)$ containing a sub-STS$(33)$ and a disjoint KTS$(9)$.  
We fill groups of the frame with this design, pulling out the common KTS$(9)$ on new elements.  This yields a KTS$(132r+75)$ containing a sub-STS$(66r+33)$.  Letting $q=11r+4$, so that $6q+9=66r+33$, we place this design on the 9th group.  Note that if we remove the common sub-KTS$(27)$ and sub-KTS$(9)$ subdesigns used on the first 8 groups, then this design in the 9th group need not have those subdesigns.

The result is a KTS of order $2u+9$ having a sub-STS$(u)$, where $u=378t + 54s + 66r + 33 = 
6(63t+9s+11r+4) + 3$.

A TD$(9,9t)$ exists for all integers $t \ge 55$; see \cite[III.3.87]{Handbook}.  A straightforward computer-assisted check shows that every integer $k \ge 3571$ can be represented as $63t+9s+11r+4$ for some $t \ge 55$, $0 \le s \le t$, and $r \in \{2,3,4,\dots,12\}$.  This implies our construction succeeds for the claimed lower bound $u \ge 6\cdot3571+3 = 21429$.
\end{proof}

\rk
Again, this is merely a preliminary bound obtained from a concise set of available designs.  The construction succeeds in practice for many smaller values of $u$.  In particular, Proposition~\ref{WFC-triple} nearly completely settles the case $u \equiv 9 \pmod{18}$.

Next, we consider general gap size $w$.

\begin{thm}
\label{general-gaps}
Suppose $w \equiv 1$ or $3 \pmod{6}$.  There exists a KTS$(2u+w)$ containing a sub-STS$(u)$ for all $u \equiv w \pmod{6}$, $u \ge 24w^2+O(w)$.
\end{thm}

\begin{proof}
The result follows for $w \in \{1,3,7,9\}$ from Theorems~\ref{thm:mkts} and \ref{thm:2u+3}, and Propositions~\ref{gap7} and \ref{gap9}.  So we assume $w \ge 13$.

The rest of the proof divides into cases.  Suppose first that $w \equiv 1 \pmod{6}$.  
By Lemma~\ref{flat} and the remark following it, there exists a KTS of order $(12s+3)w$ containing a sub-STS$((6s+1)w)$ for any positive integer $s$.
By Lemma~\ref{sip1}(a), there exists a KTS of order $(6t+1)(w+1)+1=6(w+1)t+w+2$ containing a subdesign of order $3(w+1)t+1$ for any positive integer $t$.  
Since $w$ is odd, it is relatively prime to $4(w+1)$.  From the lower bound on $u$, we can write $u=6k+w$, where $k=4(w+1)t+ws$ for positive integers $s,t$.  (Some mild restrictions on $s,t$ are detailed below, but these incur at most a linear increase to the Frobenius bound.)

As in Proposition~\ref{gap7}, we start with a TD$(9,n)$. Truncate all but $p$ points in the last group and give weights $(6;12)$ to all remaining points.  Replace 
the blocks of sizes in $\{8,9\}$ with Kirkman frames having subdesigns of type $(6;12)^x$ for $x \in \{8,9\}$. The result is a Kirkman frame  with a sub-GDD of type $(6n;12n)^8(6p;12p)^1$.
Now let  $n = t(w+1)/2$ and $6p+1=(6s+1)w$.  Add a set of $w+2$ new elements, one of which is to be allocated to the subdesign.   Fill groups with KTS having subdesigns so that, additionally, on the first 8 groups, a common KTS subdesign of order $w+2$ is aligned on the new points.  The resulting KTS has order $8 \times 6(w+1)t + (12s+3)w=2u+w$ and a subdesign of order $8 \times 3(w+1)t+(6s+1)w=u$, as required.

In the case $w \equiv 3\pmod{6}$, $w \ge 15$, we proceed similarly, using Lemma~\ref{sip1}(b) with two different choices for $z$.  We add a set of $w+z$ new points with 
$z$ of them allocated to the subdesign. Put $6n=3t(w+z)$ and $6p+z = 3s(w+z')+z'$ for positive integers $s,t$ and $z,z'$ as in the lemma.  
Since $z$ and $z'$ can be chosen so that $\gcd(8(w+z),w+z')=6$, we have a representation $u=3((8t(w+z)+s(w+z'))+z'$ as needed for all $u \ge 24w^2+O(w)$.
\end{proof}

\subsection{Smaller subdesigns}

For relatively small subdesign orders, the existence of KTS having subdesigns follows from earlier work.

\begin{prop}
\label{small-subdesigns}
There exists a KTS$(v)$ containing an STS$(u)$ as a subdesign if
\begin{itemize}[topsep=0pt]
\item
$v \ge 3u$, for $u \equiv 3 \pmod{6}$, and
\item
$v \ge 6u+3$, for $u \equiv 1 \pmod{6}$.
\end{itemize}
\end{prop}

\begin{proof}
The first case follows as a direct result of existence \cite{RS,Stinson-KTS} of KTS$(v)$ containing Kirkman subsystems.  For the second case, let $w=2u+1$.  By Theorem~\ref{thm:mkts}, there exists a KTS$(w)$, call it $\kappa$, containing a subdesign STS$(u)$.  Since $v \ge 3w$, there exists a KTS$(v)$ containing a Kirkman subsystem KTS$(w)$.  Replace this subsystem with a copy of $\kappa$.  The resulting KTS$(v)$ is still resolvable and it contains the STS$(u)$ in (the copy of) $\kappa$ as a subdesign.
\end{proof}

Frames with sub-frames can lead to some additional existence results. This type of 
construction depends on the existence of `small' KTS with STS subdesigns. 
In this case, we use KTS$(21)$ which have STS subdesgins of orders 9 and 7, respectively. 

\begin{thm}
\label{general-gaps-tripling}
(a) For all $u \equiv 3 \pmod{6}$, $u \ge 27$, there exists a KTS$(3u-6)$ containing a sub-STS$(u)$.
(b) For all $u \equiv 1 \pmod{6}$, $u \ge 25$, there exists a KTS$(3u)$ containing a sub-STS$(u)$.
\end{thm}

\begin{proof}
First note that a Kirkman frame with sub-GDD of type $(6;18)^n$ exists for all integers $n \ge 4$.  To see this, we simply apply the tripling construction to a Kirkman 
frame of type $6^n$ and view it as a sub-GDD.  
Add a set of $3$ new elements $\{\infty_1,\infty_2,\infty_3\}$.  

(a)
Fill groups with KTS$(21)$ having a subdesign of order 9, where a common block is used on 
$\{\infty_1,\infty_2,\infty_3\}$. This produces a KTS$(18n+3)$ having a sub-STS$(6n+3)$.

(b)
Fill groups with KTS$(21)$ having a subdesign of order 7 and where a common block $\{\infty_1,\infty_2,\infty_3\}$ intersects each subdesign in one element, say, $\infty_1$.
The result is a KTS$(18n+3)$ having a sub-STS$(6n+1)$.
\end{proof}

Our results leave a gap when the order of the KTS is in the intermediate range between double and triple the order of the subdesign.  Kirkman frames with subdesigns achieving a given ratio of orders may be useful in future work.  Even so, however, new ideas are likely needed to fully settle Stinson's problem of constructing KTS with STS subdesigns, or even to reduce it to a finite list of exceptions.

\section{Conclusion}

We have obtained many existence results for Kirkman triple systems containing subdesigns.  Small subdesigns can be embedded using Proposition~\ref{small-subdesigns}.  
We have shown that maximal subdesigns have connections to other topics, including strong starters and Howell designs.  Theorems~\ref{thm:mkts} and \ref{thm:2u+3} nearly completely settle the existence question in each of two different congruence classes.  More generally, we showed that KTS whose subdesign orders are a fixed offset from maximality exist with only finitely many possible exceptions; explicit bounds for two small offsets appear in Propositions~\ref{gap7} and \ref{gap9}.

We have also introduced the idea of frames with sub-GDDs. 
Many of our results are carried by direct computations of Kirkman frames with sub-GDDs
combined with recursive constructions for such designs. 
So far, these frames with sub-GDDs have landed in two special  categories in which the 
groups of the subdesign are either a union of  groups of the frame or comprise a constant 
proportion of each group. Our existence results for such frames have also been 
restricted to uniform frames. 
Breaking free of these constraints and considering more general types of 
sub-GDDs is likely one of the  next steps in our research.  

There remain a number of interesting and difficult computational problems in finding KTS with STS  subdesigns.  We can construct KTS$(v)$ with all possible admissible STS$(u)$ for the two smallest cases, $v=15$ and $v=21$. 
But, although it is easy to construct KTS$(27)$ with sub-STS$(u)$ for $u=3,9,13$,  it 
appears difficult to find a KTS$(27)$ with an STS$(7)$ subdesign.  Similarly, the 
existence of KTS$(v)$ with an STS$(7)$ subdesign is  unknown for $v=33$ and $39$. 
New direct constructions are needed for these cases. 
$\Diamond$-designs are also difficult to construct since they require both KTS and STS 
subdesigns which intersect in a smaller STS. As we indicated in Section~\ref{sec:2u+3},
 the existence of $\Diamond$-designs could be  used together with 
Kirkman frames with sub-GDDs to help compelte the spectrum.  

Another useful tool in constructing subdesigns is the ability to remove 
a subdesign and replace it with another. We used this idea in 
Proposition~\ref{small-subdesigns}. 
However,  in general, replacing a Steiner triple system in a KTS will require finding another 
resolution for the resulting STS.  Our direct constructions  use a cyclic group; this 
means that the STS subdesign is generated cyclically.  For example, there is a KTS$(39)$ 
which contains a cyclic STS$(19)$ as a subdesign.  Although, there exists an STS$(19)$ with 
an STS$(7)$ subdesign, we cannot swap the designs and expect to retain the resolution.  It 
would be nice to be able to alter the direct construction so that the resulting design admits 
further subdesigns.

Although many challenging open cases remain, we are hopeful that creative new constructions might arise for resolvable designs with subdesigns.  We offer one such example of an ad-hoc construction. 

\begin{ex}
\label{ad-hoc}
We construct a Kirkman frame of type $8^{13}$ which contains as a subdesign 
a $3$-GDD of type $8^3$.
We begin by listing the blocks of a PBD$(13,\{3,4\})$ defined on the set 
$V= \{ 1_i,2_i,3_i,4_i : i=1,2,3 \} \cup \{\infty\}$. 
The blocks are:

\begin{tabular}{lll}
$A_1 = \{1_1,1_2,1_3\}$ & 
$A_2 =\{2_1,2_2,2_3\}$ & 
$A_3 = \{3_1,3_2,3_3\}$ \\
$A_4 =\{1_1,3_3,4_2\}$ &
$A_5 = \{1_2,2_3,4_1\}$ & 
$A_6 = \{2_1,3_2,4_3\}$ \\
$A_7 = \{1_1,2_2,4_3\}$ & 
$A_8 = \{1_3,3_2,4_1\}$ & 
$A_9 = \{2_3,3_1,4_2\}$ \\
$A_{10} = \{1_2,3_1,4_3\}$ &
$A_{11} = \{1_3,2_1,4_2\}$ &
$A_{12} = \{2_2,3_3,4_1\}$\\
\end{tabular}

and 

\begin{tabular}{llll}
$B_1 = \{1_1,2_3,3_2,\infty\}$ & 
$B_2 =\{1_2,2_1,3_3, \infty\}$ & 
$B_3 = \{1_3,2_2,3_1,\infty\}$ & 
$B_4 = \{4_1,4_2,4_3,\infty\}$\\
$B_5 = \{1_1,2_1,3_1,4_1\}$ &
$B_6 = \{1_2,2_2,3_2,4_2\}$ &
$B_7= \{1_3,2_3,3_3,4_3\}$.
\end{tabular}

Replace each block of size 4 with a Kirkman frame of type $8^4$.  Each block 
$a = \{x_1,x_2,x_3,x_4\}$ is replaced by a frame with groups $x_i \times \{1,2,\dots,8\}$ 
for $i=1,2,3,4$.  Let $R_j(a,x_i)$ denote the 4 partial resolution classes associated with 
the hole $x_i \times \{1,2,\dots,8\}$,  $j=1,2,3,4$. 

\begin{figure}[htbp]
\begin{center}
\begin{tabular}{cl}
Frame hole & Resolution classes, $j=1,2,3,4$ \\
\hline
$1_1 \times \{1,2,\dots,8\}$ &  $R_j(B_5, 1_1) \cup R_j(B_7,2_3) \cup R_j(B_6,3_2)  \cup R_j(B_1,1_1)$\\
$1_2\times \{1,2,\dots,8\}$  & $R_j(B_6, 1_2) \cup R_j(B_5,2_1) \cup R_j(B_7,3_3) \cup R_j(B_2,1_2)$ \\
$1_3\times \{1,2,\dots,8\}$  & $R_j(B_7,1_3) \cup R_j(B_6,2_2) \cup R_j(B_5,3_1) \cup R_j(B_3,1_3)$ \\
$2_1\times \{1,2,\dots,8\}$ & $M_j(A_7) \cup M_j(A_8) \cup M_j(A_9) \cup R_j(B_2,2_1)$ \\
$2_2\times \{1,2,\dots,8\}$ &   $M_j (A_4) \cup M_j(A_5) \cup M_j(A_6) \cup R_j(B_3,2_2)$ \\
$2_3\times \{1,2,\dots,8\}$ & $M_j(A_{10}) \cup M_j(A_{11}) \cup M_j(A_{12}) \cup R_j(B_1,2_3)$ \\
$3_1\times \{1,2,\dots,8\}$ & $M_{4+j}(A_4) \cup M_{j+4}(A_5) \cup M_{j+4}(A_6) \cup R_j(B_3,3_1)$ \\
$3_2\times \{1,2,\dots,8\}$ & $M_{4+j}(A_{10}) \cup M_{j+4}(A_{11}) \cup M_{j+4}(A_{12}) \cup R_j(B_1,3_2)$ \\
$3_3\times \{1,2,\dots,8\}$ & $M_{4+j}(A_7) \cup M_{j+4}(A_8) \cup M_{j+4}(A_9) \cup R_j(B_2,3_3)$ \\
$4_1\times \{1,2,\dots,8\}$  & $R_j(B_5,4_1) \cup R_j(B_6,4_1) \cup R_j(B_7,4_1) \cup R_j(B_4,4_1)$ \\
$4_2\times \{1,2,\dots,8\}$  &  $M_j(A_1) \cup M_j(A_2) \cup M_j(A_3) \cup R_j(B_4,4_2)$ \\
$4_3\times \{1,2,\dots,8\}$  & $M_{j+4}(A_1) \cup M_{j+4}(A_2) \cup M_{j+4}(A_3) \cup R_j(B_4,4_3)$ \\
$\infty\times \{1,2,\dots,8\}$ & $R_j(B_1,\infty) \cup R_j(B_2,\infty) \cup R_j(B_3,\infty) \cup R_j(B_4,\infty)$ 
\end{tabular}
\end{center}
\caption{Classes for a Kirkman frame of type $8^{13}$ having a sub-GDD of type $8^3$.}
\label{8-13}
\end{figure}

Replace each block of size 3, say $b= \{y_1,y_2,y_3\}$, with a set of 3 MOLS  with 
groups $y_i \times \{1,2,\dots,8\}$.  This set of MOLS provides 8 resolution 
classes for the elements $\{y_1,y_2,y_3\} \times \{1,2,\dots,8\}$.  Label these 
8 classes $M_j(b)$ for $j=1,2,\dots,8$.

A frame resolution of the resulting blocks defined on $V \times \{1,2,\dots,8\}$ is shown in Figure~\ref{8-13}.  Notice that each block of size 3 in the underlying PBD induces 
a 3-GDD of type $8^3$ as a subdesign in this Kirkman frame of type $8^{13}$. 
From this  frame,
we obtain a KTS of order $8\cdot 13 +1=105$ which contains as a subdesign an STS$(25)$.  

\end{ex}

The construction of Example~\ref{ad-hoc} can be generalized. If there exists a Kirkman frame of type $m^4$ and $3$ MOLS of order $m$, then there is a Kirkman frame of type $m^{13}$ which 
contains as a subdesign a $3$-GDD of type $m^3$. To get a KTS, we need 
$m\equiv 2 \pmod{6}$.  A further generalization to the underlying PBD would be interesting as a topic for future work.

\section*{Acknowledgments}

Research of Peter Dukes is supported by NSERC grant 312595--2017.

\section*{Appendix: Small direct constructions}

Each of the following frames with sub-GDDs is presented either as a strong starter together with a set of triples, or as a set of base blocks with subdesign using primed and unprimed elements.  In either case, we develop blocks under the corresponding cyclic group.

\begin{itemize}

\item $(g;h)^u = (1,2)^{115}$

\noindent
\small
strong starter:
$\{102, 103\}$, $\{70, 72\}$, $\{5, 8\}$, $\{35, 39\}$, $\{84, 89\}$, $\{40, 46\}$, $\{87, 94\}$, 
$\{83, 91\}$, $\{34, 43\}$, $\{31, 41\}$, $\{48, 59\}$, $\{7, 110\}$, $\{99, 112\}$, $\{10, 24\}$, 
$\{37, 52\}$, $\{76, 92\}$, $\{32, 49\}$, $\{12, 109\}$, $\{9, 28\}$, $\{78, 98\}$, $\{26, 47\}$, $\{79, 101\}$, $\{22, 45\}$, $\{29, 53\}$, $\{57, 82\}$, $\{55, 81\}$, $\{68, 95\}$, $\{65, 93\}$, $\{38, 67\}$, $\{3, 88\}$, $\{30, 114\}$,  $\{58, 90\}$, $\{11, 44\}$, $\{2, 36\}$, $\{16, 96\}$, $\{14, 50\}$, $\{69, 106\}$, $\{27, 104\}$, $\{4, 80\}$, $\{60, 100\}$, $\{23, 64\}$, $\{71, 113\}$, $\{13, 85\}$, $\{61, 105\}$, $\{6, 51\}$, $\{17, 86\}$, $\{19, 66\}$, $\{15, 63\}$, $\{42, 108\}$, $\{25, 75\}$, $\{33, 97\}$,  $\{21, 73\}$, $\{54, 107\}$, $\{1, 62\}$, $\{56, 111\}$, $\{18, 74\}$,
$\{20, 77\}$

\noindent
triples:
$\{28,62,12\}$, $\{14,70,106\}$, $\{88,50,1\}$, $\{71,76,8\}$, $\{109,79,39\}$, $\{5,4,110\}$, $\{108,11,7\}$, $\{93,26,99\}$, $\{42,111,113\}$, $\{104,49,22\}$, $\{56,31,19\}$, $\{114,101,10\}$, $\{80,41,95\}$, $\{68,17,36\}$, $\{9,30,83\}$, $\{75,32,40\}$, $\{20,23,3\}$, $\{54,25,47\}$, $\{102,44,18\}$

\item $(g;h)^u = (2;4)^7$

\noindent
\small
$\{6'_0, 5'_1, 2_0\},
\{3'_0, 4'_0, 1_0\},
\{6'_1, 2'_1, 4_1\},
\{2'_0, 5'_0, 1_1\},
\{1'_1, 3'_1, 6_0\},
\{1'_0, 4'_1, 3_0\},
\{2_1, 4_0, 5_0\},
\{6_1, 5_1, 3_1\},$\\
$\{6'_1, 5'_0, 2_1\},
\{3'_0, 1'_1, 5_1\},
\{2'_1, 3'_1, 4_0\},
\{1'_0, 5'_1, 6_1\},
\{6'_0, 4'_0, 5_0\},
\{4'_1, 2'_0, 3_1\},
\{6_0, 2_0, 1_1\},
\{4_1, 1_0, 3_0\}$
\normalsize

\item $(g;h)^u = (3;6)^5$

\noindent
\small
$\{4_0',3_0',1_1\},
\{2_0',1_1',4_2\},
\{1_0',3_2',2_2\},
\{2_2',4_1',3_0\},
\{3_1',2_1',4_1\},
\{1_2',4_2',3_1\},
\{1_2,2_0,3_2\},
\{2_1,1_0,4_0\},$\\
$\{4_1',1_0',2_0\},
\{2_0',4_0',3_1\},
\{4_2',3_1',2_1\},
\{3_2',1_1',4_1\},
\{1_2',2_1',4_2\},
\{3_0',2_2',1_0\},
\{2_2,1_1,3_2\},
\{1_2,4_0,3_0\},$\\
$\{1_0',2_2',4_2\},
\{3_1',2_0',4_0\},
\{4_2',3_2',1_0\},
\{1_2',3_0',2_2\},
\{4_0',1_1',3_0\},
\{4_1',2_1',3_2\},
\{1_1,2_0,4_1\},
\{1_2,2_1,3_1\}$
\normalsize

\item $(g;h)^u = (3;6)^9$

\noindent
\small
strong starter:
$\{23,24\}, \{19,21\}, \{2,5\}, \{4,8\}, \{7,12\}, \{11,17\}, \{13,20\}, \{25,6\},  \{16,26\}, \{3,14\},$\\ 
$\{10,22\}, \{15,1\}$

\noindent
triples:
$\{8,14,3\}, \{11,25,21\}, \{24,23,26\}, \{22,10,2\}$

\normalsize

\item $(g;h)^u = (3;6)^{11}$

\noindent
\small
strong starter:
$\{12,13\}, \{30,32\}, \{9,6\}, \{23,27\}, \{15,20\}, \{4,10\}, \{28,2\}, \{26,1\}, \{16,25\}, \{21,31\},$\\
$\{7,19\}, \{5,18\}, \{3,17\}, \{14,29\}, \{8,24\}$

\noindent
triples:
$\{21,18,5\}, \{16,6,1\}, \{24,12,31\}, \{13,7,9\}, \{28,3,4\}$

\normalsize

\item $(g;h)^u = (3;6)^{15}$

\noindent
\small
strong starter:
$\{18,19\},\{32,34\},\{2,5\},\{17,21\},\{26,31\},\{43,4\},\{33,40\},\{20,28\},\{16,25\},\{42,7\},$\\
$\{1,12\},\{44,11\},\{41,9\},\{24,38\},\{13,29\},\{22,39\},\{35,8\},\{36,10\},\{3,23\},\{6,27\},$\\
$\{37,14\}$

\noindent
triples:
$\{11,29,34\},\{39,27,19\},\{8,22,24\},\{36,32,35\},\{18,25,31\},\{40,23,14\},\{9,20,44\}$

\normalsize

\item $(g;h)^u = (3;6)^{17}$

\noindent
\small
strong starter:
$\{36,37\},\{5,7\},\{3,6\},\{11,15\},\{42,47\},\{44,50\},\{31,38\},\{19,27\},\{9,18\},\{20,30\},$\\
$\{10,21\},\{23,35\},\{39,1\},\{49,12\},\{25,40\},\{16,32\},\{41,8\},\{14,33\},\{28,48\},\{43,13\},\{24,46\},$\\
$\{22,45\},\{2,26\},\{4,29\}$

\noindent
triples:
$\{6,39,42\},\{23,35,37\},\{1,2,8\},\{45,29,4\},\{20,15,44\},\{13,32,36\},\{30,21,41\},\{24,11,3\}$

\normalsize

\item $(g;h)^u = (3;6)^{23}$

\noindent
\small
strong starter:
$\{27,28\},\{39,41\},\{52,55\},\{49,53\},\{56,61\},\{1,7\},\{58,65\},\{12,20\},\{33,42\},$\\
$\{54,64\},\{13,24\},\{22,34\},\{25,38\},\{43,57\},\{32,47\},\{3,19\},\{4,21\},\{11,29\},\{67,17\},$\\
$\{16,36\},\{9,30\},\{44,66\},\{63,18\},\{59,15\},\{48,5\},\{35,62\},\{51,10\},\{8,37\},\{45,6\},$\\
$\{40,2\},\{68,31\},\{50,14\},\{26,60\}$

\noindent
triples:
$\{27,36,9\},\{67,65,3\},\{21,1,57\},\{66,16,50\},\{2,19,13\},\{34,20,24\},\{60,29,68\},$\\
$\{4,26,58\},\{14,59,62\},\{47,7,35\},\{43,44,18\}$

\normalsize

\item $(g;h)^u = (3;6)^{27}$

\noindent
\small
strong starter:
$\{47,48\},\{40,42\},\{22,25\},\{57,61\},\{34,39\},\{26,32\},\{1,8\},\{58,66\},\{11,20\},$\\
$\{33,43\},\{65,76\},\{55,67\},\{74,6\},\{37,51\},\{79,13\},\{75,10\},\{18,35\},\{23,41\},\{30,49\},$\\
$\{77,16\},\{69,9\},\{63,4\},\{45,68\},\{72,15\},\{59,3\},\{62,7\},\{36,64\},\{17,46\},\{80,29\},$\\
$\{78,28\},\{70,21\},\{50,2\},\{19,53\},\{38,73\},\{24,60\},\{56,12\},\{14,52\},\{5,44\},\{31,71\}$

\noindent
triples:
$\{8,75,34\},\{45,70,33\},\{71,42,39\},\{35,77,5\},\{17,22,55\},\{51,61,59\},\{36,23,40\},$\\
$\{74,50,16\},\{38,57,56\},\{20,13,48\},\{24,18,2\},\{29,65,44\},\{15,26,46\}$

\normalsize

\item $(g;h)^u = (3;6)^{29}$

\noindent
\small
strong starter:
$\{20,21\},\{25,27\},\{45,48\},\{3,7\},\{69,74\},\{12,18\},\{10,17\},\{5,13\},\{70,79\},$\\
$\{34,44\},\{80,4\},\{30,42\},\{6,19\},\{43,57\},\{24,39\},\{37,53\},\{68,85\},\{46,64\},\{56,75\},$\\
$\{51,71\},\{31,52\},\{14,36\},\{59,82\},\{23,47\},\{38,63\},\{55,81\},\{76,16\},\{60,1\},\{54,84\},$\\
$\{41,72\},\{83,28\},\{2,35\},\{61,8\},\{67,15\},\{62,11\},\{49,86\},\{40,78\},\{26,65\},\{33,73\},$\\
$\{9,50\},\{77,32\},\{66,22\}$

\noindent
triples:
$\{79,32,74\},\{45,36,46\},\{85,68,9\},\{55,75,53\},\{12,33,20\},\{76,57,42\},\{15,77,47\},$\\
$\{34,60,67\},\{7,71,21\},\{43,39,8\},\{28,64,16\},\{65,81,38\},\{40,2,86\},\{17,80,11\}$

\normalsize

\item $(g;h)^u = (3;6)^{33}$

\noindent
\small
strong starter:
$\{77,78\},\{74,76\},\{59,62\},\{91,95\},\{2,7\},\{29,35\},\{38,45\},\{97,6\},\{81,90\},\{30,40\},$\\
$\{60,71\},\{34,46\},\{24,37\},\{69,83\},\{94,10\},\{86,3\},\{48,65\},\{18,36\},\{63,82\},\{68,88\},\{5,26\},$\\
$\{42,64\},\{93,17\},\{32,56\},\{25,50\},\{28,54\},\{20,47\},\{75,4\},\{85,15\},\{11,41\},\{27,58\},\{89,22\},$\\
$\{53,87\},\{14,49\},\{79,16\},\{43,80\},\{84,23\},\{31,70\},\{98,39\},\{67,9\},\{13,55\},\{57,1\},\{8,52\},$\\
$\{51,96\},\{72,19\},\{73,21\},\{44,92\},\{12,61\}$

\noindent
triples:
$\{50,47,34\},\{84,65,59\},\{77,6,23\},\{13,42,3\},\{55,17,43\},\{28,19,30\},\{71,16,40\},\{49,44,81\},$\\
$\{39,92,69\},\{26,74,18\},\{86,90,27\},\{96,62,15\},\{36,29,78\},\{45,25,10\},\{97,98,20\},\{35,21,93\}$

\normalsize

\item $(g;h)^u = (3;6)^{39}$

\noindent
\small
strong starter:
$\{108,109\},\{35,37\},\{91,94\},\{81,85\},\{82,87\},\{69,75\},\{33,40\},\{72,80\},\{113,5\},$\\
$\{86,96\},\{10,21\},\{95,107\},\{25,38\},\{18,32\},\{47,62\},\{11,27\},\{60,77\},\{55,73\},\{44,63\},\{110,13\},$\\
$\{105,9\},\{2,24\},\{41,64\},\{92,116\},\{17,42\},\{53,79\},\{7,34\},\{48,76\},\{100,12\},\{20,50\},\{23,54\},$\\
$\{61,93\},\{65,98\},\{102,19\},\{88,6\},\{15,51\},\{8,45\},\{36,74\},\{26,66\},\{70,111\},\{89,14\},\{58,101\},$\\
$\{103,30\},\{67,112\},\{99,28\},\{59,106\},\{4,52\},\{22,71\},\{83,16\},\{115,49\},\{68,3\},\{31,84\},\{43,97\},$\\
$\{1,56\},\{90,29\},\{57,114\},\{46,104\}$

\noindent
triples:
$\{90,55,84\},\{102,28,17\},\{80,5,61\},\{83,32,29\},\{67,51,79\},\{82,9,96\},\{108,75,58\},$\\
$\{111,3,86\},\{88,81,12\},\{14,69,45\},\{18,19,113\},\{97,95,25\},\{34,13,8\},\{116,106,98\},\{87,60,22\},$\\
$\{24,104,44\},\{36,40,76\},\{89,74,21\},\{30,43,101\}$

\normalsize

\item $(g;h)^u = (3;6)^{51}$

\noindent
\small
strong starter:
$\{72,73\}, \{5,7\}, \{112,115\}, \{28,32\}, \{135,140\}, \{82,88\}, \{118,125\}, \{67,75\}, \{92,101\}, $\\ $\{49,59\}, \{69,80\}, \{33,45\}, \{47,60\}, \{8,22\}, \{83,98\}, \{15,31\}, \{106,123\}, \{114,132\}, \{107,126\},$ \\$ \{116,136\}, \{130,151\}, \{105,127\}, \{76,99\}, \{46,70\}, \{142,14\}, \{77,103\}, \{139,13\}, \{11,39\}, \{9,38\}, $\\ $\{122,152\}, \{128,6\}, \{57,89\}, \{4,37\}, \{19,53\}, \{56,91\}, \{120,3\}, \{150,34\}, \{133,18\}, \{23,62\}, \{137,24\},$ \\ $ \{104,145\}, \{12,54\}, \{43,86\}, \{50,94\}, \{29,74\}, \{78,124\}, \{141,35\}, \{36,84\}, \{95,144\}, \{96,146\}, $ \\ $\{27,79\}, \{40,93\}, \{17,71\}, \{10,65\}, \{117,20\}, \{90,147\}, \{52,110\}, \{41,100\}, \{61,121\}, \{87,148\}, \{1,63\}, $ \\ $\{134,44\}, \{55,119\}, \{66,131\}, \{2,68\}, \{111,25\}, \{81,149\}, \{16,85\}, \{109,26\}, \{42,113\}, \{129,48\}, $ \\ $\{138,58\}, \{143,64\}, \{108,30\}, \{21,97\}$
 
\noindent
triples:
$\{32,131,119\}, \{55,38,68\}, \{113,1,87\}, \{56,45,114\}, \{26,71,115\}, \{97,7,111\}, \{33,4,36\}, $ \\ $\{58,98,132\}, \{139,104,39\}, \{140,37,18\}, \{92,69,117\}, \{19,15,95\}, \{130,91,148\}, \{11,105,35\}, $ \\$ \{34,109,94\}, \{73,110,100\}, \{10,62,143\}, \{53,48,20\}, \{63,42,124\}, \{61,5,67\}, \{112,57,65\}, $ \\ $\{83,126,125\}, \{52,16,14\}, \{13,150,6\}, \{59,127,81\}$
 
\normalsize
 
\item $(g;h)^u = (3;6)^{59}$

\noindent
\small
strong starter:
$\{130,131\}, \{136,138\}, \{73,76\}, \{151,155\}, \{14,19\}, \{20,26\}, \{53,60\}, \{121,129\}, \{79,88\},$\\ 
$\{149,159\}, \{4,15\}, \{116,128\}, \{160,173\}, \{164,1\}, \{43,58\}, \{7,23\}, \{18,35\}, \{147,165\}, \{142,161\},$\\ 
$\{55,75\}, \{81,102\}, \{77,99\}, \{17,40\}, \{108,132\}, \{95,120\}, \{5,31\}, \{127,154\},
\{24,52\}, \{45,74\},$\\
$\{54,84\}, \{119,150\}, \{57,89\}, \{11,44\}, \{83,117\}, \{36,71\}, \{50,86\}, \{125,162\}, \{29,67\}, \{166,28\},$\\
$\{82,122\}, \{152,16\}, \{30,72\}, \{144,10\}, \{135,2\}, \{56,101\}, \{63,109\}, \{98,145\}, \{32,80\}, \{51,100\},$\\
$\{107,157\}, \{123,174\}, \{115,167\}, \{25,78\}, \{87,141\}, \{93,148\}, \{170,49\}, \{12,69\}, \{6,64\}, \{96,156\},$\\
$\{153,37\}, \{113,175\}, \{27,90\}, \{42,106\}, \{61,126\}, \{158,47\}, \{176,66\}, \{104,172\}, \{65,134\}, \{41,111\},$\\
$\{114,8\}, \{13,85\}, \{143,39\}, \{137,34\}, \{94,169\}, \{92,168\}, \{62,139\}, \{68,146\}, \{33,112\}, \{91,171\},$\\
$\{105,9\}, \{21,103\}, \{97,3\}, \{163,70\}, \{140,48\}, \{38,124\}, \{46,133\}, \{22,110\}$
 
\noindent
triples:
$\{78,116,140\}, \{142,18,160\}, \{159,58,21\}, \{121,47,74\}, \{71,143,43\}, \{147,12,61\}, \{54,95,34\},$\\
$\{173,163,133\}, \{109,31,29\}, \{4,88,175\}, \{169,115,123\}, \{90,45,153\}, \{94,93,50\}, \{174,170,89\},$\\
$\{20,72,41\}, \{158,60,48\}, \{8,166,40\}, \{125,68,141\}, \{150,39,62\}, \{108,52,134\}, \{44,69,127\},$\\
$\{16,25,3\}, \{155,14,7\}, \{32,49,82\}, \{80,77,9\}, \{91,106,161\}, \{35,1,164\}, \{139,79,144\}, \{15,128,26\}$
 
\normalsize
\item
$(g;h)^u = (6;12)^5$

\noindent
\small
$\{8', 7', 29\},
\{17', 19', 1\},
\{3', 6', 14\},
\{27', 1', 28\},
\{28', 22', 11\},
\{11', 4', 7\},
\{26', 18', 24\},
\{23', 14', 2\},$\\
$\{13', 24', 17\},
\{21', 9', 8\},
\{12', 29', 6\},
\{16', 2', 18\},
\{13, 12, 4\}, 
\{19, 26, 22\}, 
\{21, 23, 9\}, 
\{16, 27, 3\}$
\normalsize

\item
$(g;h)^u = (6;12)^8$

\noindent
\small
strong starter:
$\{47,4\}, \{44,21\}, \{26,5\}, \{25,10\}, \{43,9\}, \{29,28\}, \{30,19\}, \{18,20\}, \{6,36\}, \{13,41\},$\\
$\{15,22\}, \{39,45\}, \{12,2\}, \{46,17\}, \{35,23\}, \{14,31\}, \{37,34\}, \{1,27\}, \{7,11\}, \{33,42\}, \{3,38\}$

\noindent
triples:
$\{13,43,44\}, \{11,2,22\}, \{39,34,46\}, \{12,33,47\}, \{30,20,5\}, \{29,7,26\}, \{25,19,21\}$
\normalsize

\item
$(g;h)^u = (6;12)^9$

\noindent
\small
strong starter:
$\{38,19\}, \{46,39\}, \{2,17\}, \{34,48\}, \{50,37\}, \{30,8\}, \{21,47\}, \{14,52\}, \{26,3\}, \{10,22\},$ \\
$\{40,29\}, \{33,16\}, \{4,1\}, \{53,23\}, \{6,5\}, \{7,28\}, \{12,41\}, \{44,49\}, \{32,42\}, \{15,35\}, \{13,11\}, \{24,20\},$\\
$\{43,51\},\{25,31\}$

\noindent
triples:
$\{42,17,34\}, \{4,25,30\}, \{46,26,16\}, \{41,37,43\}, \{6,47,7\}, \{10,52,21\}, \{48,51,13\}, \{23,8,1\}$
\normalsize

\item
$(g;h)^u = (12;24)^4$\\
\small
$\{46', 47', 17\}, \{23', 25', 38\}, \{13', 10', 39\}, \{39', 34', 13\}, \{15', 21', 18\}, \{26', 19', 33\}, \{18', 9', 7\}, \{33', 43', 26\},$\\
$\{38', 27', 29\}, \{7', 42', 37\}, \{11', 45', 46\}, \{17', 2', 11\}, \{6', 37', 23\}, \{1', 31', 6\}, \{41', 22', 47\}, \{14', 35', 25\},$\\ 
$\{29', 3', 2\}, \{5', 30', 15\}, \{31, 30, 41\}, \{42, 35, 21\}, \{45, 19, 22\}, \{34, 5, 3\}, \{9, 27, 14\}, \{43, 1, 10\}$
\normalsize

\item
$(g;h)^u = (12;24)^5$

\noindent
\small
strong starter:
$\{58, 59\}, \{46, 48\}, \{38, 41\}, \{32, 36\}, \{28, 34\}, \{7, 14\}, \{2, 54\}, \{1, 52\},
 \{18, 29\}, \{21, 33\},$\\
$\{24, 37\}, \{12, 26\}, \{3, 19\}, \{8, 51\}, \{4, 22\}, \{16, 57\},  \{23, 44\}, \{31, 53\},
\{6, 43\}, \{11, 47\}, \{13, 39\},$\\
$\{9, 42\}, \{17, 49\}, \{27, 56\}$

\noindent
triples:
$\{17, 28, 9\}, \{48, 11, 27\}, \{16, 43, 14\}, \{3, 37, 41\}, \{42, 33, 39\}, \{36, 29, 12\}, \{31, 44, 32\}, \{46, 4, 18\}$
\normalsize

\item
$(g;h)^u = (12;24)^8$

\noindent
\small
strong starter:
$\{89, 90\}, \{33, 35\}, \{11, 14\}, \{9, 13\}, \{94, 3\}, \{59, 65\}, \{66, 73\}, \{49, 58\},
\{36,46\}, \{67,78\}, $\\
$\{43, 55\}, \{17, 30\}, \{70,84\}, \{19,34\}, \{22, 39\}, \{77,95\},
\{68,87\}, \{21,41\},\{53,74\},\{76,2\}, \{93, 20\}, $\\
$\{25,50\}, \{37,63\}, \{1, 28\}, \{23, 51\}, \{85,18\},\{27,57\}, \{71, 6\}, \{75,12\}, \{52,86\}, \{47, 82\}, \{45, 81\},$\\
$\{7, 44\}, \{62,4\}, \{15, 54\}, \{60,5\}, \{92, 38\}, \{26,69\}, \{83,31\}, \{61, 10\},
\{79, 29\}, \{91, 42\}$

\noindent
triples:
$\{21, 94, 20\},\{45, 54, 79\},\{92, 35, 23\},\{27, 44, 86\},\{9, 55, 90\},\{6, 36, 89\},
\{81, 85, 67\},$\\
$\{19, 57, 13\},\{70, 73, 63\},\{10, 5, 38\},\{41, 60, 39\},\{3, 14, 50\},
\{93, 26, 52\},\{15, 91, 46\}$
\end{itemize}

\normalsize
Following are two Kirkman triple systems with subdesigns used in our recursive constructions.
\begin{itemize}
\item
KTS$(21)$ with STS$(9)$ as a subdesign (see also \cite{CCIL,KO})

\noindent
Design:
$\{0,1,3\}$ developed (mod $7$), then expanded using an RTD$(3,3)$.

\noindent
Resolution:

\small
$\{0_1, 1_0, 3_2\}$, 
  $\{1_1, 2_1, 4_1\}$, 
  $\{2_0, 3_1, 5_2\}$, 
  $\{3_0, 4_0, 6_0\}$, 
  $\{4_2, 5_1, 0_0\}$, 
  $\{5_0, 6_1, 1_2\}$, 
  $\{6_2, 0_2, 2_2\}$

 $\{0_0, 1_2, 3_1\}$, 
  $\{1_0, 2_0, 4_0\}$, 
  $\{2_2, 3_2, 5_2\}$, 
  $\{3_0, 4_2, 6_1\}$, 
  $\{4_1, 5_1, 0_1\}$, 
  $\{5_0, 6_2, 1_1\}$, 
  $\{6_0, 0_2, 2_1\}$

 $\{0_2, 1_1, 3_0\}$, 
  $\{1_0, 2_2, 4_1\}$, 
  $\{2_1, 3_1, 5_1\}$, 
  $\{3_2, 4_0, 6_1\}$, 
  $\{4_2, 5_0, 0_1\}$, 
  $\{5_2, 6_2, 1_2\}$, 
  $\{6_0, 0_0, 2_0\}$

 $\{0_2, 1_2, 3_2\}$, 
  $\{1_0, 2_1, 4_2\}$, 
  $\{2_0, 3_0, 5_0\}$, 
  $\{3_1, 4_0, 6_2\}$, 
  $\{4_1, 5_2, 0_0\}$, 
  $\{5_1, 6_1, 1_1\}$, 
  $\{6_0, 0_1, 2_2\}$

 $\{0_0, 1_0, 3_0\}$, 
  $\{1_2, 2_0, 4_1\}$, 
  $\{2_2, 3_1, 5_0\}$, 
  $\{3_2, 4_2, 6_2\}$, 
  $\{4_0, 5_1, 0_2\}$, 
  $\{5_2, 6_0, 1_1\}$, 
  $\{6_1, 0_1, 2_1\}$

 $\{0_1, 1_1, 3_1\}$, 
  $\{1_2, 2_2, 4_2\}$, 
  $\{2_1, 3_0, 5_2\}$, 
  $\{3_2, 4_1, 6_0\}$, 
  $\{4_0, 5_0, 0_0\}$, 
  $\{5_1, 6_2, 1_0\}$, 
  $\{6_1, 0_2, 2_0\}$

 $\{0_0, 1_1, 3_2\}$, 
  $\{1_2, 2_1, 4_0\}$, 
  $\{2_2, 3_0, 5_1\}$, 
  $\{3_1, 4_2, 6_0\}$, 
  $\{4_1, 5_0, 0_2\}$, 
  $\{5_2, 6_1, 1_0\}$, 
  $\{6_2, 0_1, 2_0\}$

 $\{0_1, 1_2, 3_0\}$, 
  $\{1_1, 2_2, 4_0\}$, 
  $\{2_0, 3_2, 5_1\}$, 
  $\{3_1, 4_1, 6_1\}$, 
  $\{4_2, 5_2, 0_2\}$, 
  $\{5_0, 6_0, 1_0\}$, 
  $\{6_2, 0_0, 2_1\}$

 $\{0_2, 1_0, 3_1\}$, 
  $\{1_1, 2_0, 4_2\}$, 
  $\{2_1, 3_2, 5_0\}$, 
  $\{3_0, 4_1, 6_2\}$, 
  $\{4_0, 5_2, 0_1\}$, 
  $\{5_1, 6_0, 1_2\}$, 
  $\{6_1, 0_0, 2_2\}$

$\{0_0,0_1,0_2\}$,
$\{1_0,1_1,1_2\}$,
$\{2_0,2_1,2_2\}$,
$\{3_0,3_1,3_2\}$,
$\{4_0,4_1,4_2\}$,
$\{5_0,5_1,5_2\}$,
$\{6_0,6_1,6_2\}$,

\normalsize

\medskip
\item
KTS$(75)$ with disjoint STS$(33)$ and KTS$(9)$ subdesigns

\noindent
Design: developed as translates (mod 33) of the base blocks\\
\underline{$\{0,11,22\}$},  \underline{$\{0,10,14\}$},
$\{18, 13, 30\}$, $\{23, 5, 25\}$, $\{11, 17, 8\}$, $\{0, 1, 26\}$,
\underline{$\{0',11',22'\}$}, \underline{$\{0',1',5'\}$}
$\{\infty_1, 25', 32\}$,
$\{\infty_2, 19', 6\}$,  
$\{\infty_3, 32', 16\}$,
$\{\infty_4, 12', 28\}$,  
$\{\infty_5, 0', 14\}$,
$\{\infty_6, 28', 27\}$,
$\{\infty_7, 8', 29\}$, 
$\{\infty_8, 11', 12\}$,
$\{\infty_9, 17', 10\}$,
$\{29', 27', 19\}$,
$\{13', 16', 21\}$, 
$\{10', 4', 22\}$, 
$\{5', 31', 9\}$, 
$\{30', 22', 24\}$,
$\{23', 14', 20\}$,
$\{3', 26', 3\}$, 
$\{21', 9', 7\}$,
$\{20', 7', 2\}$,
$\{15', 1', 4\}$,
$\{6', 24', 15\}$,
$\{2', 18', 31\}$

\noindent
Resolution: the underlined blocks define four parallel classes with KTS$(9)$ on $\{\infty_1,\dots,\infty_9\}$; the remaining blocks define 33 parallel classes under translates (mod 33).
\end{itemize}



\begin{thebibliography}{99}

\bibitem{Anderson}
I. Anderson, Combinatorial Designs: Construction Methods, Ellis Horwood Limited, 
Chichester, England, 1990.

\bibitem{CDLL}
J.H.~Chan, P.J.~Dukes, E.R.~Lamken and A.C.H. Ling, Asymptotic existence of
resolvable group divisible designs.
\emph{J. Combin. Des.} 21 (2013) 112--126.

\bibitem{CCIL}
M.B. Cohen, C.J. Colbourn, L.A. Ives and A.C.H. Ling,
Kirkman triple systems of order 21 with nontrivial automorphism group.
\emph{Mathematics of Computation} 71 (2002), 873--881. 

\bibitem{Handbook}
C.J.~Colbourn and J.H.~Dinitz, eds., The CRC Handbook of Combinatorial
Designs, 2nd edition, CRC Press, Boca Raton, 2006.

\bibitem{CR}
C.J.~Colbourn and A.~Rosa, {\em Triple Systems}, Oxford Univ. Press, 1999.

\bibitem{Bluebk}
J.H.~Dinitz and D.R.~Stinton, ``Room Squares and Related Designs'' in 
Contemporary Design Theory: A Collection of Surveys, J.H.~Dinitz and D.R.~Stinson 
(Editors), John Wiley \& Sons, (1992), 137--204.

\bibitem{DinStin-hill}
J.H.~Dinitz and D.R.~Stinson, A fast algorithm for finding strong starters.
\emph{SIAM J. Alg. Disc. Meth.} 2 (1981), 50--56.

\bibitem{DinStin-howellhill}
J.H.~Dinitz and D.R.~Stinson, A note on Howell designs of odd side.
\emph{Utilitas Math.} 18 (1980), 207--216. 

\bibitem{DW}
J.~Doyen and R.M.~Wilson,
Embeddings of Steiner triple systems. 
\emph{Discrete Math.} 5 (1973), 229--239.

\bibitem{DLL}
P.J.~Dukes, E.R.~Lamken and A.C.H.~Ling, An existence theory for incomplete designs.
\emph{Canad. Math. Bull.} 59 (2016), 287--302.

\bibitem{Hanani}
H.~Hanani, Balanced incomplete block designs and related designs.
\emph{Discrete Math.} 11 (1975), 255--360.

\bibitem{KO}
J.I. Kokkala and P.R.J. \"Osterg\aa rd, Kirkman triple systems with subsystems. 
\emph{Discrete Math.} 343 (2020), 111960, 8 pp.

\bibitem{Kirkman-STS}
T.P. Kirkman, On a Problem in Combinations. 
\emph{The Cambridge and Dublin Mathematical Journal} II (1847), 191--204.

\bibitem{Kirkman-schoolgirl}
T.P. Kirkman, Query VI.
\emph{Lady's and Gentleman's Diary} 147 (1850), 48.


\bibitem{MSVW-frames}
R.C.~Mullin, P.J.~Schellenberg, S.A.~Vanstone, and W.D.~Wallis,
On the existence of frames.
\emph{Discrete Math.} 37 (1981)79--104.

\bibitem{MSV}
R.C. Mullin, D.R. Stinson, and S.A. Vanstone, Kirkman triple systems containing maximum subdesigns. 
\emph{Utilitas Math.} 21 (1982), 283--300.

\bibitem{MV}
R.C. Mullin and S.A. Vanstone, Steiner systems and Room squares.
\emph{Annals of Discrete Math.} 7 (1980) 95--104. 

\bibitem{RCWil-school}
D.K.~Ray-Chaudhuri and R.M.~Wilson, Solution of Kirkman's school-girl problem.
Proc. Symp. Pure Math. 19 (1971), Amer. Math. Soc., Providence, RI, 187--203.

\bibitem{RS}
R. Rees and D.R. Stinson, 
On the existence of Kirkman triple systems containing Kirkman subsystems.
\emph{Ars Combin.} 26 (1988), 3--16. 

\bibitem{Stinson-hill}
D.R.~Stinson, Hill-climbing algorithms for the construction of combinatorial designs.
\emph{Ann. Discrete Math.} 26 (1985), 321--334.

\bibitem{Stinson-KTS}
D.R. Stinson, A survey of Kirkman triple systems and related designs. 
\emph{Discrete Math.} 92 (1991), 371--393.

\bibitem{Stinson-Kirkframe}
D.R.Stinson,  Frames for Kirkman systems.
\emph{Discrete Math.} 65 (1987), 289--300.

\bibitem{Wilsonconst}
R.M.~Wilson,
Constructions and uses of pairwise balanced designs.
in Mathematisch Centre Tracts, vol 55, Mathematisch Centrum, Amersterdam,
1974, 18--41.

\end{thebibliography}
\end{document}